\documentclass[11pt]{amsart}
\usepackage[a4paper]{geometry}
\usepackage{amssymb,enumerate,colonequals}
\usepackage{multirow}
\usepackage{mathrsfs} 
\usepackage[section]{placeins}
\usepackage{MnSymbol}
\usepackage{tabularray}
\usepackage{extarrows}
\usepackage{lscape}
\usepackage[all,cmtip]{xy}
\usepackage[OT2,T1]{fontenc}
\usepackage{color}

\usepackage[
        colorlinks, citecolor=darkgreen,
        backref,
        pdfauthor={Elvira Lupoian},
]{hyperref}

\usepackage{comment}
\usepackage{multirow}
\usepackage[table]{xcolor}
\usepackage{tabularray}
\usepackage{longtable}
\usepackage{hyperref}
\hypersetup{
    colorlinks=true,
    linkcolor=darkgreen,
    filecolor=magenta,      
    urlcolor=darkgreen,
    pdftitle={Overleaf Example},
    pdfpagemode=FullScreen,
    }

\usepackage{amsmath}
\usepackage{amsthm}
\usepackage{amssymb}
\usepackage{tikz-cd}
\usepackage{tikz}
\usepackage{tabularx}
\usepackage{hyperref}

\newcommand{\Q}{\mathbb{Q}}
\newcommand{\R}{\mathbb{R}}
\newcommand{\Z}{\mathbb{Z}}

\DeclareMathOperator{\divv}{div}
\DeclareMathOperator{\Div}{Div}

\DeclareMathOperator{\num}{num}

\DeclareMathOperator{\Gal}{Gal}

\DeclareMathOperator{\lcm}{lcm}

\DeclareMathOperator{\Divv}{Div}

\DeclareMathOperator{\Circ}{Circ}

\newcommand{\tors}{{\operatorname{tors}}}
\newcommand{\cusps}{{\operatorname{cusps}}}

\newcommand{\SL}{\operatorname{SL}}

\newcommand{\floor}[1]{\left\lfloor #1 \right\rfloor}

\numberwithin{equation}{section}

\newtheorem{thm}{Theorem}[section]
\newtheorem{lem}[thm]{Lemma}
\newtheorem{prop}[thm]{Proposition}
\newtheorem{cor}[thm]{Corollary}

\theoremstyle{definition}

\newtheorem*{Ack}{Acknowledgements}

\theoremstyle{remark}

\theoremstyle{definition}

\theoremstyle{remark}
\newtheorem{remark}[equation]{Remark}

\newenvironment{psmallmatrix}
  {\left(\begin{smallmatrix}}
  {\end{smallmatrix}\right)}

\definecolor{darkgreen}{rgb}{0,0.5,0}

\DeclareRobustCommand{\SkipTocEntry}[5]{}

\begin{document}
\title{Cuspidal Class Number Formulae for Intermediate Modular Curves of Prime Level}

\begin{abstract}
Let $X$ be a modular curve lying between $X_{0}(p)$ and $X_{1}(p)$, with $p \ge 13$ prime. The cuspidal group of $X$ is the subgroup of the Jacobian generated by equivalence classes of differences of cusps. In this paper, we give explicit formulae for the order of the cuspidal group, and two subgroups of it defined over the rationals.
\end{abstract}

\author{Elvira Lupoian}

\address{Department of Mathematics, University College London, London, WC1H 0AY, UK}

\email{e.lupoian@ucl.ac.uk}
\date{\today}
 \maketitle

 \tableofcontents
\section{Introduction}
The \textit{cuspidal group} associated to a modular curve is the subgroup of the Jacobian generated by linear equivalence classes of differences of cusps. The Manin-Drinfeld theorem \cite{Manin}, \cite{Drinfeld} show that the cuspidal group is finite when the modular curve is associated to a \textbf{congruence} subgroup of $\SL_{2}(\Z)$. This result is not constructive and gives no method for computing the cuspidal group or even its order, known as the \textit{cuspidal class number}.

The cuspidal class number, or certain factors of it, have been determined for some families of modular curves, see \S\ref{lit} for an overview of literature. In this paper, we give explicit formulae for the order of the cuspidal group, and of certain rational subgroups of it, for all  modular curves lying between $X_{0}(p)$ and $X_{1}(p)$, where $p \ge 13$ is prime.

\subsection{Main Results} 
Let $p \ge 13$ be a prime number and set $G :=  \left( \Z/p \Z \right)^*$. We take $H$ to be any proper subgroup of $G$ which contains $-1$\footnote{Note that $p=13$ is the first prime for which a \textbf{proper} subgroup of $G$ containing $-1$ exists.}, and we write  $G_H$ for the quotient $G/H$ and  $\hat{G}_H$ for the character group of $G_H$.  Let $X_H(p)$ be the complete modular curve defined by $H$, see \S\ref{covers} for precise definitions. We note that $X_{\{ \pm 1\}}(p) = X_1(p)$ and $X_G(p) = X_0(p)$, and that there are natural coverings of curves 
$$ X_1(p) \rightarrow X_H(p) \rightarrow X_0(p).$$
We refer to such $X_H(p)$ as \textit{intermediate modular curves}. We write $\mathcal{C}_{H}(p)$ for the cuspidal subgroup of $X_{H}(p)$ and $h_{H}(p)$ for its order.

For any Dirichlet character modulo $p$, $\chi$, the second generalised Bernoulli number $B_{2,\chi}$ is defined as 
 $$B_{2, \chi} = p \sum_{a=1}^{p} \chi(a) B_{2}( \frac{a}{p})$$
where $B_{2}(x) = x^2 -x + \frac{1}{6}$ is the second Bernoulli polynomial. 
\begin{thm} \label{thm1}
Let $p \ge 13$ be a prime number and $H \subseteq ( \Z / p\Z)^{*}$, a proper subgroup  with $-1 \in H$. The cuspidal class number of $X_{H}(p)$ is
 \begin{center}
   $ h_{H}(p) = \num ( \frac{p-1}{12k}) \left( \displaystyle \prod_{\chi} \frac{1}{4} B_{2, \chi} \right)^{2}$   
 \end{center}
where the product is taken over all non-principal characters in $\hat{G}_H$. 
 \end{thm}

  Throughout this paper we work with subgroup $H$ which contain $-1$ since $X_{H}(p) = X_{\langle H, \pm 1 \rangle}(p)$, as explained in \S\ref{pres}. Moreover, we assume that $H$ is a proper subgroup since $H = \{ \pm 1 \}$ corresponds to $X_1(p)$, and the cuspidal class number was determined by Takagi \cite{takagi1993cuspidal}  and $H =G$, the modular curve is $X_0(p)$, and the cuspidal group was proved to be cyclic of order the numerator of $\frac{p-1}{12}$ by Ogg \cite{Ogg}.

As we assume that $-1 \in H$, all Dirichlet characters appearing in the formula for $h_H(p)$ are even, and the generalized Bernoulli number simplifies in this case, namely, 

$$ B_2, \chi = p \displaystyle \sum_{a =1 }^p \chi(a) \left( \frac{a}{p} \right)^2$$
for $\chi \ne 1$. 
The cusps of $X_{H}(p)$ fall into a natural partition, the $k$ cusps $P_{0}, \ldots, P_{k-1}$ lying above $c_{\infty}$, the infinity cusp of  $X_{0}(p)$, and the $k$ cusps $Q_0, \ldots Q_{k-1}$ lying above $c_{0}$, the zero cusp of  $X_{0}(p)$. We work with a model of $X_{H}(p)$ for which the $P_{i}$ are rational and the $Q_{j}$ are defined over a subfield of $ \Q(\zeta_{p})^{+}$ and form one orbit under the action of Galois. There are three natural subgroup of $\mathcal{C}_{H}(p)$ 
defined over $\Q$ to consider. We write $\mathcal{C}_{H}^{\infty}(p)$ for the subgroup generated by linear equivalence class of differences of the rational cusps $P_{0}, \ldots, P_{k-1}$, $\mathcal{C}_{H}^{\Q}(p)$ for the subgroup generated by linear equivalence classes of degree $0$ cuspidal divisors fixed by the action of Galois and $\mathcal{C}_{H}(p)(\Q)$ for the subgroup of $\mathcal{C}_{H}(p)$ fixed by the Galois action, that is, linear equivalence classes of cuspidal divisors $[D]$ such that $D^{\sigma} \sim D$ for all $\sigma \in \Gal( \Bar{\Q} / \Q)$. By definition, they satisfy the following natural containment 
\begin{center}
    $\mathcal{C}_{H}^{\infty}(p) \subseteq \mathcal{C}_{H}^{\Q}(p) \subseteq \mathcal{C}_{H}(p)(\Q) \subseteq \mathcal{C}_{H}(p)$.
\end{center}
We write $h_{H}^{\infty}(p)$ and $h_{H}^{\Q}(p)$ for the orders of $\mathcal{C}_{H}^{\infty}(p)$ and $\mathcal{C}_{H}^{\Q}(p)$ respectively.

\begin{thm} \label{thm2}
Let $p \ge 13$ be a prime number and $H \subset ( \Z / p\Z)^{*}$, a proper subgroup with $-1 \in H$. Then: 
\begin{center}
    $h_{H}^{\infty}(p) =  \displaystyle \prod_{\chi} \frac{1}{4}B_{2,\chi} $
\end{center}
where the product is taken over all non-principal characters in $\hat{G}_H$. 
\end{thm}

In the case of $H = \{ \pm 1 \}$, where $X_H(p) = X_1(p)$, the three groups above are equal, see \cite{lupoian2025modularunitsx1p}.  This is no longer the case for general $H$, and in fact, $\mathcal{C}_H^\infty(p) \ne \mathcal{C}_H^\Q(p)$, unless $k \in 
\{ \frac{p-1}{4}, \frac{p-1}{6}, \frac{p-1}{12} \} \cap \Z $, see Proposition \ref{specialk}. However, $\mathcal{C}_H^\Q(p) = \mathcal{C}_H(p)(\Q)$ for all prime numbers $p \ge 5$, and we determine its order in the following theorem.  

 \begin{thm} \label{thm3}
  Let $p \ge 13$ be a prime number and  $H \subset ( \Z / p\Z)^{*}$, a proper subgroup with $-1 \in H$. Then the order of $\mathcal{C}_{H}^{\Q} (p)$ is given by 
  \begin{center}
      $ h_{H}^{\Q}(p) =  \num ( \frac{p-1}{12k}) \left( \displaystyle \prod_{\chi} \frac{1}{4} B_{2, \chi} \right) $
      \end{center}
where the product is taken over all non-principal characters in $\hat{G}_H$. 
  \end{thm}

Regarding the structure of the cuspidal group, we can give an explicit decomposition in terms of certain subgroups of it.  We write $\mathcal{C}^0_H(p)$ for the subgroup of the cuspidal group supported on the cusps $Q_0, \ldots,Q_{k-1}$. Note that this subgroup is of order $h_H^\infty(p)$, since the Frike involution $w_p$ defines a bijection between $\mathcal{C}_H^\infty(p)$ and $\mathcal{C}_H^0(p)$. We write $D_{H.p}$ for the cuspidal divisor of order the numerator of $\frac{p-1}{12k}$ defined in \S\ref{structuresection}. 
\begin{thm} \label{thm4}
  Let $p \ge 13$ be a prime number and  $H \subset ( \Z / p\Z)^{*}$, a proper subgroup of index $k$ with $-1 \in H$. Then 
  $$ \mathcal{C}_H(p) = \mathcal{C}_H^\infty(p) \oplus \langle [D_{H,p} ] \rangle \oplus \mathcal{C}_H^0(p).$$
\end{thm}
In explicit examples, we describe the structure of the cuspdial group completely, see Theorem \ref{thm5} and \S\ref{comps}.
 
 \subsection{Method}
 The formulae appearing in Theorems \ref{thm1}, \ref{thm2} and \ref{thm3} are derived using explicit bases for certain groups of \textit{modular units}. These are functions on modular curves whose divisors are supported on cusps on only.  The cuspidal group $\mathcal{C}_{H}(p)$ is simply the quotient of $\Divv_{\text{cusps}}^{0}(X_{H}(p))$, the group of degree $0$ divisors on $X_{H}(p)$ supported on cusps, by the subgroup generated by divisors of modular units
 \begin{center}
     $\mathcal{C}_{H}(p) = \Divv_{\text{cusps}}^{0}(X_{H}(p))/ \langle \divv (f) \ \vert \ f \in \mathcal{F}_{H}(p) \rangle$
 \end{center}
where $\mathcal{F}_{H}(p)$ is the group of modular units. We note that $\mathcal{F}_{H}(p)$ is a finitely generated group, since the divisor map defines a map 
\begin{center}
$\mathcal{F}_{H}(p) \rightarrow \Divv_{\text{cusps}}^{0}(X_{H}(p))$
\end{center}
and this is injective modulo constants. Therefore, computing a basis of $\mathcal{F}_{H}(p)$, up to constants, is sufficient to determine $h_{H}(p) = \vert \mathcal{C}_{H}(p) \vert$, as given such a basis, the order of the quotient is the absolute value of the determinant the matrix whose entries are defined by the divisors of this basis. 

We construct such a basis using \textit{Siegel functions}, see \S\ref{SF} for precise definitions and \S\ref{criteria} and \S\ref{basis} for our constructions. Our strategy is to use the basis of modular units on $X_1(p)$ constructed in \cite{lupoian2025modularunitsx1p}, and we prove that the action induced from the Galois cover of curves $X_1(p) \rightarrow X_H(p)$ determines the group of modular units on $X_H(p)$ completely.  

The formulae for $h_{H}^{\infty}(p)$ and $h_{H}^{\Q}(p)$ are similar; in these cases we reduce our calculation to finding bases for the groups of modular units whose divisors are supported on cusps lying above $c_{\infty} \in X_0(p)$ and those which are fixed by the natural Galois action on divisors. 

\subsection{Applications and Observations}
\subsubsection{Generators of Cuspidal Groups and Explicit Calculations} A by-product of our method is an explicit basis of $\mathcal{F}_H(p)$, the group of modular units on $X_H(p)$. The cuspidal group is the quotient group  of the group of degree $0$ divisors on $X_H(p)$ supported on cusps by the subgroup consisting of divisors of modular units. Therefore, given the explicit description of the basis of $\mathcal{F}_H(p)$, for a given $p$ and $H$, determining the cuspidal group is a straightforward calculation. Moreover, the rational groups $\mathcal{C}_H^\infty(p)$ and $\mathcal{C}_H(p)(\Q)$, can be computed in a similar manner. 
\begin{thm} \label{thm5}
   Let $ 13 \le p \le 97$ be a prime number and $H$ a proper subgroup of  $\left( \Z / p \Z \right)^*$ such that $-1 \in H$. The groups $\mathcal{C}_H(p)$, $ \mathcal{C}_H(p)(\Q)$ and $\mathcal{C}_H^\infty(p)$ are explicitly determined in Section 8. 
\end{thm}
The prime $p=97$ is not optimal in this calculation, and cuspidal groups for larger levels can computed using our bases, as demonstrated in the case of $X_1(p)$, see \cite{lupoian2025modularunitsx1p}. 

\subsubsection{Structure of the cuspidal group}
For the modular curve $X_0(N)$ and some levels $N$, determining the group of modular units is often sufficient to determine the structure of certain subgroups of the cuspidal group abstractly. Modular units on $X_0(N)$ are constructed from quotients of Dedekind eta functions and the orders of vanishing at cusps of these functions can be manipulated algebraically. The reader should refer to the beautiful work of Yoo \cite{yoo2023rational}, for an overview of Dedekind eta functions and their role in determining the structure of the cuspidal group.

Besides $X_0(N)$, the structure of the cuspidal group is not known for other families of modular curves. In the case of the intermediate modular curves (and $X_1(N)$), the presence of Bernoulli numbers in the orders of vanishing makes manipulating divisors of modular units much harder.  Possible exceptions to this are when $k$ is small, and hence there are few modular units and few cusps.  For example, when $k=2$, there are four cusps and three modular units, and the structure of the cuspidal group was investigated in many examples in \cite{PMB_2025____97_0}.   
In future work, the author will study the cuspidal group in such cases in detail.

\subsubsection{Cyclotomic Fields and Zeta Functions}
The presence of Bernoulli numbers in the formula for $h_H(p)$ might be surprising at first. A possible explanation for their appearance is the close relation between $\mathcal{C}_H(p)$ and the torsion of the Jacobian of $X_H(p)$ over the field of definition of the cusps, proved via the `Eisenstein ideals', that is, the ideals of the Hecke algebra which annihilate  certain spaces of Eisenstein series. This connection is explored in the author's upcoming work.

The Bernoulli numbers allow us to relate $h_H(p)$ to a special value of Dedekind zeta functions. More precisely, let $K = \Q(\zeta_p)^+$ be the maximal real subfield of the cyclotomic field $\Q(\zeta_p)$.  This is a Galois extension of $\Q$ with Galois group 
$$ \Tilde{G} =  \left( \Z/ p \Z \right)^*/ \{ \pm 1 \}.$$ 
For any subgroup, $\Tilde{H}$, of $\Tilde{G}$, we write $K_H$ for the fixed field $K^{\Tilde{H}}$. 
The Dedekind zeta function of $K_H$ is 
$$ \zeta_{K_H}(s) = \displaystyle \prod_{\chi} L(s, \chi)$$
where the product is taken the character group of $\Tilde{G}/\Tilde{H}$. 
Notably, for $\chi \ne 1$, $ L(-1, \chi) = -\frac{B_{2,\chi}}{2}$ and $L(-1, 1) = \zeta(-1) = \frac{-1}{12}$, and thus 
$$ \zeta_{K_H}(-1) = \displaystyle  \frac{-1}{12}\prod_{\chi \ne 1 } - \frac{B_{2,\chi}}{2}.$$
Let $H$ be the unique lift of $\Tilde{H}$ to $G$ (containing $-1$). Combining Theorem \ref{thm1} with the above we obtain 
$$ h_H(p) = \num \left( \frac{p-1}{12k} \right) \frac{144}{2^{2k-2}} \zeta_{K_H}(-1)^2$$
where $k$ is the index of $H$ (and the degree of $K_H$). 
\subsection{Overview of Literature} \label{lit}
The cuspidal class number of $X(p^n)$,  where $p \ge 5$ is prime was computed by Kubert and Lang \cite{KL}.

For $X_0(N)$, the cuspidal group or cuspidal class number are known in the following cases.
\begin{itemize}
    \item For $N \ge 5$ and prime, Ogg \cite{Ogg} showed that the cuspidal group of $X_0(N)$ is cyclic of order the numerator of $\frac{N-1}{12}$. 
    \item For $N$ with $N$ square-free, it is computed by Takagi \cite{takagi1997cuspidal}. 
    \item  In \cite{takagi2010cuspidal}, Takagi determined the cuspidal class number for certain Atkin-Lehner quotients of $X_0(N)$. 
\end{itemize}

The structure of the cuspidal group, or certain subgroups of it, for $X_0(N)$, has been studied for various levels, \cite{yoo2023rational}, \cite{yoo2025rational}, \cite{guo2023rational}, \cite{lorenzini1995torsion}, \cite{chua1997rational}, \cite{ling1997q}, \cite{rouse2015spaces}. 

For $X_1(N)$, the cuspidal class number, or certain factors of it, is known in the following cases.
\begin{itemize}
\item When $N \ge 5$ is prime, Klimek \cite{klimek} determined the order of the subgroup generated by cusps lying above the infinity cusp on $X_0(N)$, and the order of the full cuspidal group was later determined by Takagi \cite{takagix1p}. 
    \item  For general $N \ge 5$, Yu \cite{yu1980cuspidal} computed the order of the subgroup generated by cusps lying above the infinity cusp on $X_0(N)$. This was recently generalised by Brunault \cite{brunault2026bernoulli}, who determined the order of the subgroup of the cuspidal group generated by linear equivalence classes of cuspidal divisors fixed by the Galois action. 
\item Takagi determined the cuspidal class number of $X_1(p^n)$ for $p \ge 5$ prime \cite{takagi1993cuspidal}, for $X_1(3^n)$ \cite{takagi1995cuspidal} and for $X_1(2^{2n+1})$ \cite{takagi2008cuspidal}. 
\item When $N = 2p$, where $p \ge 5$ is prime, the cuspidal class number was computed by Takagi \cite{takagi2012cuspidal}, and the orders and structures of various subgroups of the cuspidal group of $X_1(2p)$ were investigated in \cite{takagi2014ℚ}.  
\item The $p$-adic and $\ell$-valuations of the cuspidal class number for the `tower'  $X_1(Np^n)$, with $p \not \mid N$, were computed by Sun \cite{sun2010cuspidal} for $n$ sufficiently large. 
\item  The structure of the $p$-primary part of the cuspidal group of $X_1(p^n)$, for primes $p \ge 5$ we investigated by Yang and Yu \cite{yang2010structure}.

\end{itemize}

The question of determining the $h_H^\infty(p)$ and $h_H^\Q(p)$ was first considered by Chen \cite{Chen}. The method is similar to ours, namely, it constructs a basis for $\mathcal{F}_H^\infty(p)$ and $\mathcal{F}_H^\Q(p)$ using scaled Siegel functions $E_i(\tau)$\footnote{Note that the $E_i$ are different functions to the $E_i$ defined in this paper.} first introduced by Yang \cite{yang2007modular}. The proof of Chen's main result is incomplete. More precisely, in \cite[Theorem 2]{Chen} Chen claims that $e_{g} = e_{a^{ik}g}$ follows from the fact that $f(\gamma \tau)  = f(\tau)$. This is inaccurate, since $E_i$s are not independent, and certain products of them can be constant functions. For such a statement to hold, one needs to argue using a basis of the group of modular units on $X_1(p)$ instead of using the $E_i$ only, or by employing an argument similar to Proposition \ref{expsH}, with both strategies consisting of non-trivial necessary arguments.  Although, Chen's argument can be made precise, our method is more complete, since the $E_i$ considered in that paper cannot be used to construct a basis of the full space of modular units, and hence determine $h_H(p)$.

\subsection{Notation and Conventions}
Throughout this paper, we use the following notation. 
\begin{itemize}
    \item[] $p \ge 13$ : a prime number;
    \item[] $n := \frac{1}{2}(p-1)$;
\item[] $\alpha$ : a fixed integer whose reduction modulo $p$ generates $(\Z / p \Z)^{*} $;
\item[]$\beta$ : the unique $\beta \in \{ \pm 1, \pm 5 \}$ such that $p \equiv \beta \mod{12}$;
\item[] $G := \left( \Z / p \Z \right)^*$;
\item[] $H$ : a proper subgroup of $G$  which contains $-1$;
\item[] $k = [ \left( \Z / p \Z \right)^{*} : H ]$; 
\item[] $X_H(p)$: the complete modular curves associated to $H$; 
\item[] $P_{i}$ : the cusp of $X_{H}(p)$ represented by $\frac{\alpha^{i}}{p}$, with $0 \le i \le k-1$;
\item[] $Q_{i}$ : the cusp of $X_{H}(p)$ represented by $\frac{p}{\alpha^{i}}$, with $0 \le i \le k-1$;
\item[] $g_{\mathbf{a}}(\tau)$ : the Siegel function defined by $\mathbf{a} \in \Q^{2}$;
\item[] $\{ x \}$ : the fractional part of $x \in \R$;
\item[] $B_{2}(x) = x^{2} - x + \frac{1}{6}$ : the second Bernoulli polynomial;
\item[] $a_{i} = \frac{p}{2} B_{2} ( \{ \frac{\alpha^{i}}{p} \})$ for $ 0 \le i \le n-1$;
\item[] $B_{2, \chi} = p \sum_{a=1}^{p} \chi(a) B_{2}(\frac{a}{p})$ where $\chi$ is a Dirichlet character modulo $p$;
\item[] $\Div_{\text{cusps}}^{0}(X_H(p))$ : the group of degree $0$ divisors on $X_{H}(p)$, supported on cusps;
\item[] $\Div_{\Q \ \text{cusps}}^{0}(X_H(p))$ : the group of degree $0$ divisors on $X_{H}(p)$, supported on cusps and fixed by the Galois action;
\item[] $\Div_{\infty -\text{cusps}}^{0}(X_H(p))$ : the group of degree $0$ divisors on $X_{H}(p)$, supported on $P_{0}, \ldots, P_{k-1}$;
\item[] $\Div_{0 -\text{cusps}}^{0}(X_H(p))$ : the group of degree $0$ divisors on $X_{H}(p)$, supported on $Q_{0}, \ldots, Q_{k-1}$;
\item[] $\mathcal{F}_{H}(p)$ : the group of modular units on $X_{H}(p)$;
\item[] $\mathcal{F}_{H}^{\Q}(p)$ : the group of modular units on $X_{H}(p)$ whose divisors are in 
$\Div_{\Q \ \text{cusps}}^{0}(X_H(p))$;
\item[] $\mathcal{F}_{H}^{\infty}(p)$ : the group of modular units on $X_{H}(p)$ whose divisors are in $\Div_{\infty-\text{cusps}}^{0}(X_H(p))$;
\item[] $\mathcal{F}_{H}^{0}(p)$ : the group of modular units on $X_{H}(p)$ whose divisors are in $\Div_{0-\text{cusps}}^{0}(X_H(p))$;
\item[] $\mathcal{C}_{H}(p) := \Div_{\text{cusps}}(X_H(p)) / \langle \divv(f) : f \in \mathcal{F}_H(p) \rangle$: the cuspidal group;
\item[] $\mathcal{C}_{H}^{\Q}(p) := \Div_{\text{cusps}}^{\Q}(X_H(p)) / \langle \divv(f) : f \in \mathcal{F}_H^{\Q}(p) \rangle$: the subgroup of the cuspidal group supported on rational divisors;
\item[] $\mathcal{C}_{H}^{\infty}(X_H(p)) := \Div_{\text{cusps}}^{\infty}(X_H(p)) / \langle \divv(f) : f \in \mathcal{F}_H^{\infty}(p) \rangle$: the subgroup of the cuspidal group supported on cusps lying above the $\infty$ cusp of $X_0(p)$;
\item[] $\mathcal{C}_{H}^{0}(X_H(p)) := \Div_{0 - \text{cusps}}^{0}(X_H(p)) / \langle \divv(f) : f \in \mathcal{F}_H^{0}(p) \rangle$: the subgroup of the cuspidal group supported on cusps lying above the $0$ cusp of $X_0(p)$;

\item[] $h_{H}(p) = \vert \mathcal{C}_{H}(p) \vert$;
\item[] $h_{H}^{\Q}(p) = \vert \mathcal{C}_{H}^{\Q}(p) \vert$;
\item[] $h_{H}^{\infty} (p) = \vert \mathcal{C}_{H}^{\infty}(p)\vert$;
\item[] $h_{H}^{0} (p) = \vert \mathcal{C}_{H}^{0}(p)\vert$.
\end{itemize}

\subsection{Layout} This paper is organised as follows. In Section \ref{pres} we review key facts about modular curves, Siegel functions and modular units on $X_1(p)$.
In Section \ref{criteria} we give necessary and sufficient conditions for certain products of Siegel functions to be modular units on $X_H(p)$; and in Section \ref{basis} we use this criterion to construct a basis for the group of modular units. In Section \ref{cuspidalclassnumber} we use the explicit basis to determine the cuspidal class number $h_H(p)$. In Section \ref{subgroups} we find bases for certain subgroups of modular units and determine $h_H^{\Q}(p)$ and $h_H^{\infty}(p)$. In Section \ref{structuresection} we prove Theorem \ref{thm4}
In Section \ref{comps} we compute the cuspidal group and its rational subgroups for all $13 \le p \le 97$ and all possible subgroups $H$.

\begin{Ack} 
I'm extremely grateful to Samir Siksek and Damiano Testa for many helpful conversations in the initial stages of this project. I thank Hwajong Yoo for our discussions about the cuspidal group in the case where $H$ is the group of squares. The author is supported by the EPSRC Doctoral Prize fellowship EP/W524335/1 and the extension UKRI3030.
\end{Ack}

\section{Preliminaries} \label{pres}
\subsection{Intermediate Modular Curves}
\label{covers}

Let $N$ be a positive integer. The homomorphism 
$$\Gamma_0(N) \rightarrow \left( \Z/ N \Z \right)^\times, \ \begin{psmallmatrix}
      a & b \\ c & d  
\end{psmallmatrix} \mapsto a \mod{N} $$
induces an isomorphism 
$$ \Gamma_0(N)/\Gamma_1(N) \cong \left( \Z / N \Z \right)^\times .$$
An \textbf{intermediate congruence subgroup of level $N$} is a congruence subgroup $\Gamma$ of $\SL_2(\Z)$ that satisfies $\Gamma_1(N) \subseteq \Gamma \subseteq \Gamma_0(N)$. It follows that for any such $\Gamma$, there exists a unique  subgroup $H \subseteq \left( \Z / N \Z \right)^\times$ such that $$\Gamma/\Gamma_1(N) \cong H.$$
Explicitly,
$$ \Gamma = \Gamma_H(N) := \{ \begin{psmallmatrix}
     a & b \\ c & d
\end{psmallmatrix} \in \Gamma_0(N) \ \vert \ a \mod{N} \in H \}. $$
Conversely, a subgroup $H$ of $\left( \Z / N \Z \right)^*$ defines an intermediate congruence subgroup $\Gamma_H(N)$, and thus intermediate congruence subgroups of level $N$ are in bijection with subgroups of $\left( \Z / N \Z \right)^*$. 

Let $H$ be a subgroup of $\left( \Z / N \Z \right)^\times$.  
We write $X_H(N)$ for the complete modular curve associated with the congruence subgroup $\Gamma_H(N)$, and $Y_H(N)$ for the open modular curve. Notably, as $-I_2 \in \SL_2(\Z)$ acts the same as the identity, 
$$X_{H}(N) = X_{\langle H, \pm 1 \rangle}(N) \ \text{and} \ Y_{H}(N) = Y_{\langle H, \pm 1 \rangle}(N),$$
 hence, for the remainder of this paper we assume that $-1 \in H$, or equivalently, we consider subgroups of $\left( \Z / N \Z \right)^\times/ \{ \pm 1 \}$ only. 

The relation $ \Gamma_1(N) \subseteq \Gamma_H(N) \subseteq \Gamma_0(N)$ induces a natural covering of curves 
$$ X_1(N) \rightarrow X_H(N) \rightarrow X_0(N)$$
and we may view $X_H(N)$ as the quotient curve 
$$ X_H(N) = X_1(N)/H$$
where $H$ is viewed as a subgroup of $\left( \Z/ N \Z \right)^\times/ \{ \pm 1 \}$, which we identify with the Galois group of the covering of curves $X_1(N) \rightarrow X_0(N)$. This perspective gives us a natural moduli interpretation for the points of $Y_H(N)$, that is, they represent isomorphism classes 
$$ [E, [P]_H]$$
where $E$ is an elliptic curve and $[P]_H$ is the $H$-orbit of an $N$-torsion point on $E$, or more precisely, $[P]_H = [P']_H$ if and only if $P = hP'$ for some $h \in H$.  

Evident from all descriptions above is that taking $ H = \{ \pm 1 \}$ (resp. $\left( \Z / N \Z \right)^\times/ \{ \pm 1 \}$) recovers the classical modular curve $ X_1(N)$ (resp. $X_0(N)$).

The projective curve $X_{H}(N)$ has a Shimura (canonical) model over $\Q$, with its cusps being defined over the subfield of $ \Q( \zeta_{N})^{+}$, the maximal real subfield of the cyclotomic field $\Q( \zeta_{N})$, fixed by $H$. The cusps of $X_{H}(N)$ can be described explicitly. The following proposition is a direct generalisation of results in \cite{Ogg} and \cite[Chapter 3]{diamond2005first}, and we include a proof for the readers convenience.
\begin{prop} \label{cuspsprop}
 Let $N$ and $H$ be as above. The cusps of $X_{H}(N)$ are represented by $\pm \begin{psmallmatrix}
     x \\ y 
 \end{psmallmatrix}$, where $x,y \in \left( \Z / N \Z \right)^{*} / H $, or $xy=0$, $x$ is reduced modulo $\gcd(y,N)$ and $\gcd(x,y) =1$.
\end{prop}

\begin{proof}
 The cusps of $X(N)$ are represented by $\pm \begin{psmallmatrix}
     x \\ y
 \end{psmallmatrix}$, where $x,y \in \Z/ N \Z$ are relative prime. As the group $G(N) = \SL_{2}(\Z)/ \Gamma(N)$ operates naturally on  the left, the cusps of $X_{H}(N)$ are simply orbits of $G_{H}(N) = \Gamma_{H}(N)/ \Gamma(N)$. More precisely, as 
 \begin{center}
     $ G_{H}(N) = \{ \begin{psmallmatrix}
         a & b \\ 0 & d 
     \end{psmallmatrix}\ \vert \ a, d \in H, ad =1, b \in ( \Z / N 
     \Z) \}$ 
 \end{center} 
 the cusps of $X_{H}(N)$ correspond to $\{ 
 \pm \begin{psmallmatrix}
     ax + by \\ dy
 \end{psmallmatrix} \ : a,d \in H, b \in \Z / N \Z \}$ for some fixed $x,y \in \Z / N \Z$, relatively prime, and the claim follows. 
\end{proof}

\subsection{Siegel Functions} \label{SF}
Let $\mathbf{a} = \left( a_{1}, a_{2} \right) \in \Q^{2} \setminus \Z^{2}$. The \textit{Siegel functions} associated to $\mathbf{a}$ is a function on the upper half plane defined by the infinite product presentation 
\begin{center}
    $g_{\mathbf{a}} ( \tau ) = - e^{2 \pi i a_{2} (a_{1} -1)/2} q_{\tau}^{B_{2}\left( a_{1} \right)/2} \left( 1 - q_{\tau} \right) \displaystyle \prod_{n=1}^{\infty} \left( 1 - q_{\tau}^{n} q_{z} \right) \left( 1 - q_{\tau}^{n}/q_{z} \right) $
\end{center}
where $z = a_{1} \tau + a_{2}$, $q_{\tau} = e^{2 \pi i \tau}$, $q_{z} = e^{2 \pi i z}$ and $B_{2} = x^{2} - x + \frac{1}{6}$ is the second Bernoulli polynomial. The modular units constructed in the next section are products of Siegel functions, and we use the following elementary properties. 

\begin{prop} \label{trans2}
For $\mathbf{a} = \left( a_{1}, a_{2} \right) \in \Q^{2} \setminus  \Z^{2}$, $\mathbf{b} = \left( b_{1}, b_{2} \right) \in \Z^{2}$  and $\gamma = \begin{psmallmatrix}
        a & b \\ c & d \end{psmallmatrix} \in SL_{2}\left( \Z \right)$
\begin{itemize}
    \item[i)] $ g_{-\mathbf{a}} ( \tau )  =  - g_{\mathbf{a}} ( \tau )$;
    \item[ii)] $ g_{\mathbf{a} + \mathbf{b}} ( \tau )  = \epsilon \left( \mathbf{a}, \mathbf{b} \right) g_{\mathbf{a}}( \tau )$ where $\epsilon \left( \mathbf{a}, \mathbf{b} \right) = \left( -1 \right)^{b_{1}b_{2} + b_{1} + b_{2}} e^{- \pi i \left( b_{1}a_{2} - b_{2}a_{1} \right)}$;
    \item[iii)]   $g_{\mathbf{a}} \left( \gamma \tau \right) = \zeta \left( \gamma \right) g_{\mathbf{a} \gamma } ( \tau )$ where $ \zeta \left( \gamma \right)$ is a $12$th root of unity, depending only on $\gamma$.
    \end{itemize}
\end{prop}
\begin{proof}
 Parts (i) and (ii) follow from the definition of $g_{\mathbf{a}} ( \tau )$ and the results proved in \cite[Pages 27-29]{KL}. Part (iii) follows from the definition of $g_{\mathbf{a}} ( \tau )$ and the transformation formula for the Dedekind eta function proved in  \cite{newman1957construction}.
 \end{proof}

We note that for any congruence subgroup $\Gamma \subseteq \SL_{2}(Z)$, we can use the above results to derive criteria for when a product of Siegel functions is invariant under the action of $\Gamma$, and hence define functions on the corresponding complete modular curve $X_{\Gamma}$. Crucially, as argued in \cite[Pages 30-31]{KL}, $g_{\mathbf{a}}(\tau)$  has no zeros or poles in the upper half plane; and thus any product of Siegel functions which defines a function $X_{\Gamma}$ is a \textit{modular unit} on $X_{\Gamma}$, that is a function on the modular curve whose divisor is supported on cusps only.  The orders of vanishing at cusps of such functions are explicit, and can be deduced from the product expansion, as argued in \cite[Page 31]{Kubert19753}.
\begin{prop} \label{ordga}
For any  $\mathbf{a} = \left( a _{1}, a_{2} \right) \in \Q^{2} \setminus \Z^{2}$ 
\begin{center}
    $\text{ord}_{q}\left( g_{\mathbf{a}} ( \tau ) \right) = \frac{1}{2} B_{2} \left( \{ a_{1} \} \right)$
\end{center}
where $B_{2}$ is the second Bernoulli polynomial as defined above and  $\{ x \} \in \left[ 0 , 1 \right)$ denotes the fractional part of $x \in \mathbb{R}$.
\end{prop}

\subsection{Modular Units and Covering}
The arguments of the proceeding sections rely on the following observation. For any congruence subgroups $\Gamma_{1} \subseteq \Gamma_{2}$, there is an induced covering of the associated modular curves 
\begin{center}
    $\pi : X_{\Gamma_{1}} \longrightarrow X_{\Gamma_{2}}$.
\end{center}
The resulting pullback map on the associated the function fields 
\begin{center}
    $\pi^{*} : \mathbb{C} \left( X_{\Gamma_{2}} \right) \longrightarrow \mathbb{C} \left( X_{\Gamma_{1}} \right)$
\end{center}
is defined by $\pi^{*} f = f \circ \pi $, and divisors are computed using the classical formula 
\begin{center}
    $\nu_{x} \left( \pi^{*} f \right) = e_{x} \nu_{\pi \left( x \right)} \left( f \right)$
\end{center}
where $\nu_{x}$ denoted the valuation at $x \in X_{\Gamma_{2}}$ and $e_{x}$ is the ramification degree at $x$. Thus the pullback map preserves modular units and we state the following elementary result. 
\begin{lem} \label{mucovs}
For any congruence subgroups $\Gamma_{1} \subseteq \Gamma_{2}$, any modular unit on $\Gamma_{2}$ is a modular unit $f$ on $\Gamma_{1}$ such that 
\begin{center}
    $f \left( \gamma \tau \right) = f \left( \tau \right)$
\end{center}
for all $\gamma \in \Gamma_{2} / \Gamma_{1}$.
\end{lem}

For the rest of this paper, we fix a prime number $p \ge 5$ and a subgroup $H < (\Z/p\Z)^{*}$ which contains $-1$. Our strategy for computing modular units on $X_{H}(p)$ is as follows. Note that $\Gamma_{1}\left( p\right) \subseteq \Gamma_{H}\left( p\right)$ and hence by the above any modular unit on $X_{H}(p)$ is a modular unit on $X_{1}(p)$ which is additionally invariant under the action of $ \Gamma_{H}(p)/ \Gamma_{1}(p)$. We use the basis of modular units for $X_{1}(p)$ constructed in \cite{lupoian2025modularunitsx1p} to determine the modular units for $X_{H}(p)$.

\subsection{Modular Units on \texorpdfstring{$X_{1}(p)$}{k}}
We give an overview of result in \cite{lupoian2025modularunitsx1p}. For the rest of this paper we fix $\alpha$, an integer whose reduction modulo $p$ has multiplicative order $p-1$ in $\left( \Z / p \Z \right)^{*}$. This fixed choice of $\alpha$ plays a key role in all our calculations. It appears in the definitions of our functions, the criteria for invariance, the representatives of cusps and the orders of vanishing. 
Let $n = \frac{p-1}{2}$ and for $0 \le i \le n-1$ define
\begin{itemize}
    \item $E_{i} \left( \tau \right) = g_{\left( \frac{\alpha^{i}}{p}, 0 \right)} \left( \tau \right) \displaystyle \prod_{j=0}^{p-2} g_{\left( \frac{\alpha^{i}}{p}, \frac{\alpha^{j}}{p} \right)}\left( \tau \right); $
    \item $F_{i} \left( \tau \right) = g_{\left( 0, \frac{\alpha^{i}}{p} \right)} \left( \tau \right)$.
\end{itemize}
The following is proved in \cite[Theorem 1 and Proposition 9]{lupoian2025modularunitsx1p}.
\begin{prop} 
\label{X1pcongs}
Let  $f \left( \tau \right) = \displaystyle \prod_{i=0}^{n-1} E_{i}\left( \tau\right)^{e_{i}} F_{i} \left( \tau \right)^{f_{i}} $, with integer exponents $ e_{i}, f_{i} $. Then $f$ is a modular unit on $X_{1}\left( p\right)$ if and only if $e_{i}, f_{i} $ 
satisfy:
\begin{itemize}
    \item $ \displaystyle \sum_{i=0}^{n-1} e_{i} \alpha^{2i} \ \equiv  \ \displaystyle \sum_{i=0}^{n-1} f_{i} \alpha^{2i} \   \equiv 0 \ \text{mod} \ p;$
    \item $ \displaystyle \sum_{i=0}^{n-1} pe_{i} + f_{i} \ \equiv \ 0 \ \text{mod} \  12.$
\end{itemize}
 Moreover, up to scaling, all modular units on $X_{1}\left( p\right)$ are of this form. 
\end{prop}
In fact, in \cite{lupoian2025modularunitsx1p} we constructed an explicit basis for the group of modular units on $X_{1}(p)$, modulo constants, which we briefly recall here.  Fix $\beta \in \{ \pm 1, \pm 5 \}$, such that $p \equiv  \ \beta \ \mod{12}$. For $0 \le i \le n-2$ define
\begin{align}
    & A_{i} \left( \tau \right) = E_{i} \left( \tau \right) E_{n-1} \left( \tau \right)^{-\alpha^{2i+2}}F_{n-1} \left( \tau \right)^{p (\alpha^{2i+2} -1)}; \\
    & B_{i} \left( \tau \right) = F_{i}\left( \tau \right) F_{n-1}\left( \tau \right)^{ -\alpha^{2i+2} + p \beta (\alpha^{2i+2} -1)};
\end{align}
and also define 
\begin{align}
    & A_{n-1} \left( \tau \right) = E_{n-1} \left( \tau \right)^{p} F_{n-1} \left( \tau \right)^{-\beta p}; \\
    & B_{n-1} \left( \tau \right) = F_{n-1}\left( \tau \right)^{12p}.
\end{align}
The following results is proved in \cite[Theorem 1 and Corollary 17]{lupoian2025modularunitsx1p}.
\begin{thm} \label{basisx1p}
Any subset of $ \{ A_{0}, \ldots, A_{n-1}, B_{0}, \ldots, B_{n-1} \}$ consisting of $p-2$ functions and containing $A_{n-1}$ and $B_{n-1}$ is a basis for the group of modular units on $X_{1}\left( p\right)$, modulo constants.
\end{thm}
When working with products of Siegel functions which are modular units on $X_{1}(p)$ it is often useful to use both presentations, as products of $E_{i}$ and $F_{j}$ due to the symmetry of these functions, and as products of $A_{i}$ and $B_{j}$, since subsets of these form bases for the group of modular units. We state the following elementary results to aid us when switching between the two presentations. 

\begin{prop} \label{EFtoAB}
  Let $f(\tau) = \displaystyle \prod_{i=0}^{n-1}E_{i}(\tau)^{e_{i}}F_{i}(\tau)^{f_{i}}$ be a modular unit on $X_{1}(p)$, with integer exponents $e_{i},f_{j}$ satisfying the congruence conditions of Proposition \ref{X1pcongs}. Define 
  \begin{align*}
     &  a(f) = \frac{1}{p}( e_{n-1} + \displaystyle \sum_{i=0}^{n-2} \alpha^{2i+2} e_{i}); \\
     & b(f) = \frac{1}{12p} ( \beta p a(f) + f_{n-1} - \displaystyle \sum_{i=0}^{n-2}p(\alpha^{2i+2} -1)e_{i} - \displaystyle \sum_{i=0}^{n-2} f_{i}( - \alpha^{2i+2} + p \beta ( \alpha^{2i+2} -1))). 
  \end{align*}
  Then $a(f), b(f) \in \Z$ and 
  \begin{center}
      $ f( \tau) = \left( \displaystyle \prod_{i=0}^{n-2} A_{i}( \tau)^{e_{i}} \right)A_{n-1}( \tau)^{a(f)} \left( \displaystyle \prod_{i=0}^{n-2} B_{i}(\tau)^{f_{i}} \right)B_{n-1}(\tau)^{b(f)}.$
  \end{center}
\end{prop}
\begin{proof}
 The fact that $ a(f), b(f)$ are integers is a consequence of the congruences of Proposition \ref{X1pcongs}. The second claim follows from the definition of $A_{i}$ and $B_{j}$. 
\end{proof}

\begin{prop} \label{constAB}
  Let $f ( \tau ) = \displaystyle \prod_{i=0}^{n-1} A_{i}( \tau )^{a_{i}} B_{i}( \tau)^{b_{i}}$ with $a_{i}, b_{i} \in \Z $. Then $f$ is a constant function on $X_{1}\left( p \right)$ if and only if there exists $a \in \Z $ such that   
  \begin{align*}
      & a_{i} = b_{i} = a \ \text{for all} \  0 \le i \le n-2 \\
      & a_{n-1} = -an_{1} \\
      & b_{n-1} = -an_{2} 
  \end{align*}
  where $n_{1} =  \frac{1}{p} \frac{\alpha^{p-1} -1 }{1 - \alpha^{2}}$ and $n_{2} = \frac{1}{12} ( \beta n_{1} + n_{1} -p(\beta +1)n_{1} - ( \beta + 1) - ( \beta + 1)(n -1))$.
\end{prop}
\begin{proof}
 Consider the function 
 \begin{center}
     $f_{1} \left( \tau \right) = \displaystyle \prod_{i=0}^{n-1} E_{i} \left( \tau \right) F_{i} \left( \tau \right)$.
 \end{center}
 By Proposition \ref{X1pcongs} this is a modular unit on $X_{1}\left( p \right)$ and we verify using \cite[Proposition 13]{lupoian2025modularunitsx1p} that the order of vanishing at each cusp is zero.  Therefore, this function is necessarily constant (and non-zero) and we write $C$ for its value. By the definition of $A_{i}$ and $B_{j}$ we observe that
 \begin{align*}
     f_{2}\left( \tau \right) & = \displaystyle \prod_{i=0}^{n-2} A_{i} \left( \tau \right) B_{i} \left( \tau \right) \\
     & = C A_{n-1} \left( \tau \right)^{n_{1}} B_{n-1} \left( \tau \right)^{n_{2}}
 \end{align*}
 where $n_{1}$ and $n_{2}$ are as above, and we note that they are integers by the definition of $\alpha$ and $\beta$. Thus we have an expression for $A_{0}(\tau)$
 \begin{center}
     $A_{0}(\tau) = C \displaystyle \prod_{i=1}^{n-2}A_{i}(\tau)^{-1}\prod_{i=0}^{n-2}B_{i}(\tau)^{-1} A_{n-1} \left( \tau \right)^{n_{1}} B_{n-1} \left( \tau \right)^{n_{2}}$.
 \end{center}
This allows us to rewrite any $f ( \tau ) = \displaystyle \prod_{i=0}^{n-1} A_{i}( \tau )^{a_{i}} B_{i}( \tau)^{b_{i}}$
as 
\begin{center}
    $f( \tau) = C \displaystyle \prod_{i=1}^{n-2} A_{i}( \tau )^{a_{i} - a_{0}} \prod_{i=0}^{n-2} B_{i}( \tau)^{b_{i} - a_{0}} A_{n-1}(\tau)^{n_{1}a_{0} + a_{n-1}} B_{n-1}(\tau)^{n_{2}a_{0} + b_{n-1}}$
\end{center}
and since $A_{1}, \ldots, A_{n-1}, B_{0}, \ldots, B_{n-1}$ is a basis for the group of modular units modulo constants the claim follows. 
\end{proof}

In \cite{lupoian2025modularunitsx1p} we also give a basis for the group of modular units whose divisors are supported in cusps lying above $\infty i \in X_{0}(p)$, and for those whose divisors are fixed by the Galois action. The following propositions are simple consequences of \cite[Theorems 2 and 3]{lupoian2025modularunitsx1p}.

\begin{prop} \label{F1pinfbasis}
Let $f \left( \tau \right) = \displaystyle \prod_{i=0}^{n-1} E_{i} \left( \tau \right)^{e_{i}}$ for some $e_{i} \in \Z $. Then $f$ is a modular unit on $X_{1}\left( p \right)$ whose divisor is supported on $\infty-$cusps if the exponents $e_{i}$ satisfy 
\begin{itemize}
    \item $ \displaystyle \sum_{i=0}^{n-1} e_{i} \alpha^{2i} \  \equiv 0 \ \text{mod} \ p;$
    \item $ \displaystyle \sum_{i=0}^{n-1} e_{i} = 0 .$
\end{itemize}
Moreover, up to scaling, all such modular units are of the above form. 
\end{prop}

\begin{prop} \label{F1pQbasis}
Let $f \left( \tau \right) = \displaystyle \prod_{i=0}^{n-1} E_{i} \left( \tau \right)^{e_{i}}$ for some $e_{i} \in \Z $. Then $f$ is a modular unit on $X_{1}\left( p \right)$ whose divisor is fixed by the action of Galois if the exponents $e_{i}$ satisfy 
\begin{itemize}
    \item $ \displaystyle \sum_{i=0}^{n-1} e_{i} \alpha^{2i} \  \equiv 0 \ \text{mod} \ p;$
    \item $ \displaystyle \sum_{i=0}^{n-1} e_{i} \equiv 0 \mod{12}.$
\end{itemize}
Moreover, up to scaling, all such modular units are of the above form. 
\end{prop}

\section{Modular Units on \texorpdfstring{$X_{H}(p)$}{k}}
\label{criteria}
For the rest of this paper, we fix a prime number $p \ge 5$ and a proper subgroup $H \subset \left( \Z / p \Z \right)^{*}$, such that $-1 \in H$. Let $\alpha$ be as in the previous section, that is, an integer whose reduction modulo $p$ generates $( \Z / p \Z )^{*}$. It follows that $H$ is generated by $( \alpha \mod{p} )^{k}$ for some $1 \le k \le n = \frac{p-1}{2}$ and $k$ divides $n$. For $0 \le i \le k-1$ define:
\begin{itemize}
\item $  G_{i} \left( \tau \right) = \displaystyle \prod_{t=0}^{n/k-1} E_{i+ tk} \left( \tau \right); $
\item $ H_{i} \left( \tau \right) = \displaystyle \prod_{t=0}^{n/k-1} F_{i+ tk} \left( \tau \right)$. 
\end{itemize}
The main result proved in this section is the following. 

\begin{thm} \label{XHmainthm}
Let $ \{g_i, h_i\}_{0 \le i \le k-1} \subset \Z$ and set  
$$f \left( \tau \right) = \displaystyle \prod_{i=0}^{k-1} G_{i} \left( \tau \right)^{g_{i}} H_{i} \left( \tau \right)^{h_{i}}.$$ 
Then $f \left( \tau \right)$ is a modular unit on $X_{H}\left( p \right)$ if and only if the exponents satisfy
\begin{align}
    & \displaystyle \frac{n}{k} \left( \sum_{i=0}^{k-1} p g_{i} + h_{i} \right) \ \equiv 0 \mod{12}; \label{cong1XH}\\
    & \displaystyle \sum_{i=0}^{k-1} g_{i} +  h_{i} \ \equiv \ 0 \mod{2}. \label{cong2XH2}
\end{align}

\end{thm}

To prove this, we verify that such functions are modular units on $X_{1}(p)$ and additionally verify that they are invariant under the action  of $\Gamma_{H}(p)/\Gamma_{1}(p)$. We begin with some elementary results. 

\begin{lem} \label{cosetH}
Let $a, b$ be any integers such that 
 \begin{center}
     $\gamma :=  \begin{psmallmatrix}
        \alpha^{k} & bp \\ p & 2a 
    \end{psmallmatrix} \in \SL_2(\Z)$.
 \end{center}
Then
 \begin{center}
     $\Gamma_{H}\left( p \right) = \langle \gamma, \Gamma_{1}\left( p \right) \rangle$. 
 \end{center}
\end{lem}

\begin{proof}
For any $A \in \Gamma_{H} \left( p\right)$, 
\begin{center}
    $A \equiv \begin{psmallmatrix}
        \alpha^{kt} & \star \\
        0 & \delta
    \end{psmallmatrix} \ \text{mod} \ p$
\end{center}
for some $0 \le t \le \frac{n}{k}$.
If $t =0$, then 
\begin{center}
    $A \left( \gamma^{-1}\right)^{t} =  \ A \equiv \begin{psmallmatrix}
        1 & \star \\ 
        0 &  1  
    \end{psmallmatrix} \ \ \text{mod} \ p$
\end{center}
hence $A \left( \gamma^{-1} \right)^{t} \in \Gamma_{1} \left( p\right)$. If $t >0$, we note  that 
\begin{center}
    $A \gamma^{-1} \equiv \begin{psmallmatrix}
        \alpha^{k(t-1)} & \star \\
        0 & \star
    \end{psmallmatrix} \ \text{mod} \ p$
\end{center}
hence inductively, $A \left( \gamma^{-1} \right)^{t} \in \Gamma_{1} \left( p\right)$  and our claim follows.
\end{proof}

\begin{prop} \label{transEFH}
Let $\gamma$ be as defined in Lemma \ref{cosetH}. Then 
\begin{align*}
    & E_{i} \left( \gamma \tau \right) = \zeta ( \gamma )^{p} e^{\pi i \frac{\alpha^{i+k}(2a -1)}{p} \frac{1 - \alpha^{p-1}}{p(1 - \alpha)} } E_{i+k}( \tau ) \ \text{for} \ 0 \le i \le n-k -1, \\
    & E_{i} \left( \gamma \tau \right) =  -\zeta ( \gamma )^{p} (-1)^{\alpha + \alpha^{i+k -n}} e^{ -\pi i \frac{\alpha^{i+k-n}(2a+1)}{p} \frac{1 -\alpha^{p-1}}{p(1 - \alpha)} } E_{i+k -n} \left( \tau \right) \ \text{for} \ n-k \le i \le n-1, \end{align*}
    where $\zeta ( \gamma )$ is a $12$th root of unity, depending on $\gamma$ only (c.f Proposition \ref{trans2}). Similarly, 
    \begin{align*}
    & F_{i} ( \gamma \tau) = - \zeta( \gamma) (-1)^{\alpha^{i}} e^{\pi i \frac{\alpha^{n-k +2i}}{p}} F_{n-k+ i} ( \tau) \ \text{for} \ 0 \le i \le k-1,\\
     & F_{i} ( \gamma \tau ) = \zeta( \gamma) ( -1 )^{\alpha^{i-k}} e^{ -\pi i \frac{\alpha^{2i -k }}{p}} F_{i - k } \left( \tau \right) \ \text{for} \ k \le i \le n-1.
       \end{align*}
    \end{prop}

\begin{proof}
Suppose $0 \le i \le n-k -1$. We apply Proposition \ref{trans2} (ii) and (iii) to obtain 
 \begin{align*}
     g_{\left( \frac{\alpha^{i}}{p}, 0 \right)} \left( \gamma \tau \right) &  = \zeta \left( \gamma \right) g_{(0, b \alpha^{i}) + \left( \frac{\alpha^{i+k}}{p},0 \right)} \left( \tau \right) \\ 
     & = \zeta \left( \gamma \right) \left( -1 \right)^{b \alpha^{i}} e^{\pi i b \frac{\alpha^{2i+k}}{p}} g_{\left( \frac{\alpha^{i+k}}{p},0 \right)} \left( \tau \right)
     \end{align*}
and for $ 0 \le j \le p-2$ 
\begin{align*}
    g_{\left( \frac{\alpha^{i}}{p}, \frac{\alpha^{j}}{p} \right)} \left( \gamma  \tau \right) & = \zeta( \gamma) (-1)^{b \alpha^{i+j} + b \alpha^{i} + \alpha^{j}}e^{- \pi i ( \frac{2a \alpha^{2j}}{p} - \frac{b \alpha^{2i+k}}{p})} g_{( \frac{\alpha^{i+k}}{p}, \frac{2a \alpha^{j}}{p})}(\tau).
\end{align*}
Let $m_{j}$ and $n_{j}$ be the unique integers such that $ 0 \le m_j \le p-2$ and 
\begin{center}
    $2a\alpha^{j} = \alpha^{m_{j}} + n_{j}p$.
\end{center}
Then by Proposition \ref{trans2} (ii) we have 

\begin{center}
    $g_{( \frac{\alpha^{i+k}}{p}, \frac{2a \alpha^{j}}{p})}(\tau) = (-1)^{n_{j}}e^{\pi i n_{j} \frac{\alpha^{i+k}}{p}}g_{( \frac{\alpha^{i+k}}{p}, \frac{\alpha^{m_{j}}}{p})}(\tau). $
 \end{center}
From the definition of $E_{i}(\tau)$, it follows that
\begin{center}
  $E_{i}(\tau) = \zeta(\gamma)^{p}(-1)^{a_{1}}e^{- \pi i a_{2}}E_{i+k}(\tau)$
\end{center}
where $a_{1}, a_{2}$ are the constants:
\begin{align*}
    & a_{1} = b\alpha^{i} + \displaystyle \sum_{j=0}^{p-2} \left( b\alpha^{i+j} + b \alpha^{i} + \alpha^{j} + n_{j} \right), \\
    & a_{2} = - b \frac{\alpha^{2i +k}}{p} + \displaystyle \sum_{j=0}^{p-2}  \left( \frac{2a\alpha^{2j}}{p} - b \frac{\alpha^{2i+k}}{p} - n_{j} 
    \frac{\alpha^{i+k}}{p} \right) .
    \end{align*}
We observe that $a_{1}$ is an integer, and its parity is 
\begin{center}
    $ a_{1} \equiv b\alpha^i + b\alpha^i ( 1 + (p-2) \alpha) + (p-1)b\alpha^i + (1 + (p-2)\alpha) + (1 + (p-2) \alpha) \mod{2}$
\end{center}
since $n_j \equiv \alpha^{m_j} \mod{2}$ and $\sum_{0 \le j \le p-2} \alpha^j \equiv 1 + \alpha \mod{2}$. Moreover, by the definition of $\gamma$ we have  $b \equiv 1 \mod{2}$ and so our congruence becomes 
$$ a_1 \equiv \alpha \mod{2}.$$
As for $a_{2}$, we simplify the above to obtain 
\begin{center}
    $a_{2} = - b \alpha^{2i+k} + \frac{2a}{p} \displaystyle \sum_{j=0}^{p-2} \alpha^{2j} - \frac{\alpha^{i+k}}{p} \displaystyle \sum_{j=0}^{p-2} n_{j}$.
\end{center}
 Noting that $ - b \alpha^{2i+k} \equiv \alpha \mod{2}$ and that $\frac{2a}{p} \displaystyle \sum_{j=0}^{p-2} \alpha^{2j}$ is an even integer by the choice of $\alpha$, we find that 
\begin{center}
    $E_{i}( \gamma \tau) = \zeta(\gamma)^{p} (-1)^{\alpha} (-1)^{\alpha} e^{\pi i C} = \zeta(\gamma)^p e^{\pi i C}$
\end{center}
where $ C =   \frac{\alpha^{i+k}}{p} \displaystyle \sum_{j=0}^{p-2} n_{j}$.
The result follows from the definition of $n_{j}$, that is, 
$$ \displaystyle \sum_{j=0}^{p-2} n_j = \frac{1}{p} \displaystyle \sum_{j=0}^{p-2} 2a \alpha^j - \alpha^{m_j} =  \frac{1}{p} \left( 2a -1 \right) \displaystyle \sum_{j=0}^{p-2} \alpha^j = \frac{2a-1}{p} \frac{1 - \alpha^{p-1}}{1 - \alpha}.$$

For $n-k \le i \le n-1$, let $m$ be the unique integer such that $\alpha^{n} = -1 + mp$, and note that $ \alpha^{i +k} = - \alpha^{i+k-n} + \alpha^{i+k-n}m p$. Thus applying Proposition \ref{trans2} we find 
\begin{align*}
  g_{( \frac{\alpha^{i}}{p}, 0)}( \gamma \tau)  & = \zeta(\gamma)(-1)^{b\alpha^i}
e^{\pi i b \frac{\alpha^{2i+k}}{p}} g_{(\frac{\alpha^{i+k}}{p}, 0)}(\tau) \\
& = \zeta(\gamma)(-1)^{1 + b\alpha^i + \alpha^{i+k-n}m}
e^{\pi i b \frac{\alpha^{2i+k}}{p}} g_{(\frac{\alpha^{i+k -n}}{p}, 0)}(\tau).
\end{align*}

For any $0 \le j \le p-2$ we have  
\begin{align*}
   g_{( \frac{\alpha^{i}}{p}, \frac{\alpha^{j}}{p})}(\tau) & =  \zeta(\gamma)(-1)^{b \alpha^{i+j} + b\alpha^{i} + \alpha^{j}} e^{- \pi i ( \frac{2a\alpha^{2j}}{p} - b \frac{\alpha^{2i+k}}{p})} g_{( \frac{\alpha^{i +k}}{p}, \frac{2a \alpha^{j}}{p})}(\tau) \\
   & =  \zeta(\gamma)(-1)^{ 1 + m \alpha^{i+k-n} + b \alpha^{i+j} + b\alpha^{i} + \alpha^{j}} e^{- \pi i ( - \alpha^{i+k-n}m \frac{2a\alpha^{j}}{p} +  \frac{2a\alpha^{2j}}{p} - b \frac{\alpha^{2i+k}}{p})} g_{( \frac{\alpha^{i +k -n}}{p}, \frac{-2a \alpha^{j}}{p})}(\tau)
\end{align*}

Let $l_{j}$ and $k_{j}$ be the unique integers satisfying $ 0 \le l_{j} \le p-2$ and $-2a\alpha^{j} = \alpha^{l_{j}} + k_{j}p$. Then we have 
\begin{center}
   $g_{( \frac{\alpha^{i +k}}{p}, \frac{-2a \alpha^{j}}{p})}(\tau)= (-1)^{k_{j}} e^{- \pi i (-k_j )\frac{\alpha^{i+k-n}}{p}}g_{( \frac{\alpha^{i+k-n}}{p}, \frac{\alpha^{l_{j}}}{p})}( \tau) $.
\end{center}
Proceeding as before, we have 
\begin{center}
  $E_{i}(\tau) = -\zeta(\gamma)^{p}(-1)^{a_{3}}e^{- \pi i a_{4}}E_{i+k-n}(\tau)$
\end{center}
where the constants $a_{3}$ and $a_{4}$ are defined as 
\begin{align*}
    a_3 & = 1 + b \alpha^i + \alpha^{i+k-n}m + \displaystyle \sum_{j=0}^{p-2}(b\alpha^{i+j} + b\alpha^i + \alpha^j + m \alpha^{i+k -n} +1 + k_j), \\
    a_4 & = -b\frac{\alpha^{2i+k}}{p} + \displaystyle \sum_{j=0}^{p-2} \left(2a \frac{\alpha^{2j}}{p} -b \frac{\alpha^{2i+k}}{p} - \alpha^{i+k-n} \frac{2a\alpha^j}{p}m - k_j \frac{\alpha^{i+k-n}}{p}\right).
\end{align*}

Arguing as in the previous case, we find that $a_{3} \equiv 1  + \alpha^{i+k-n} \mod{2}$ and 
\begin{center}
   $a_{4} = -b \alpha^{2i+k}  +  \frac{2a}{p }\displaystyle \sum_{j=0}^{p-2} \alpha^{2j}  - \frac{2am\alpha^{i+k-n}}{p}  \displaystyle \sum_{j=0}^{p-2} \alpha^{j} - \frac{\alpha^{i+k-n}}{p} \displaystyle \sum_{j=0}^{p-2} k_{j}.$
\end{center}
As $ \frac{2a}{p }\displaystyle \sum_{j=0}^{p-2} \alpha^{2j}$ and $\frac{2am\alpha^{i+k-n}}{p}  \displaystyle \sum_{j=0}^{p-2} \alpha^{j} $ are even integers, $  -b \alpha^{2i+k} $  is an integer and $b \alpha^{2i +k} \equiv \alpha \mod{2}$, it follows that 
$$ E_i(\gamma\tau) = \zeta(\gamma)^p (-1)^{1 + \alpha^{i+k-n} }(-1)^\alpha e^{\pi i D}$$
where $D = \frac{\alpha^{i+k-n}}{p} \displaystyle \sum_{j=0}^{p-2} k_{j}$, and the resuls follows from the definition of $k_j$ are before. 

The transformation of $F_{i} \left( \tau \right) = g_{\left( 0, \frac{\alpha^{i}}{p} \right)} \left( \tau \right)$ is similar. If $k \le i  \le n-1$, we note that $2a \alpha^{i} = \alpha^{i-k} + n_{i}$
for a unique $n_{i} \in \Z$. It follows from Proposition \ref{trans2} that 
\begin{align*}
    F_{i} ( \gamma \tau) & = \zeta ( \gamma ) (-1)^{\alpha^{i}n_{i} + \alpha^{i} + n_{i}} e^{-\pi i \frac{\alpha^{2i-k}}{p}} g_{( 0, \frac{\alpha^{i-k}}{p} )} ( \tau) \\
    & =  \zeta ( \gamma ) (-1)^{\alpha^{i-k}} e^{-\pi i \frac{\alpha^{2i-k}}{p}} F_{i-k}( \tau) \end{align*}
    since $\alpha^{i} n_{i} + \alpha^{i} + n_{i} \equiv \alpha \alpha^{i-k} + \alpha + \alpha^{i-k} \equiv \alpha^{i-k} \mod{2}$.

For $0 \le i \le k-1$, let $k_{i}$ be the unique integer such that $2a \alpha^{i} = - \alpha^{n-k+1} + k_{i} p$. Applying Proposition \ref{trans2} we get 
  \begin{align*}
      F_{i} ( \gamma \tau)  
      & = \zeta \left( \gamma \right) (-1)^{\alpha^{i} k_{i} + \alpha^{i} + k_{i}} e^{\pi i \frac{\alpha^{n-k+2i}}{p}} g_{( 0, -\frac{\alpha^{n-k+i}}{p})} ( \tau ) \\
            & = - \zeta \left( \gamma \right) (-1)^{ \alpha^{i} } e^{\pi i \frac{\alpha^{n-k+2i}}{p}} F_{n-k+i} \left( \tau \right).
 \end{align*}
 This completes our proof.                                                                                 
\end{proof}

\begin{prop} \label{transGH}
Let $\gamma$ be as defined in Lemma \ref{cosetH}. Then for all $0 \le i \le k-1$ we have 
\begin{align*}
 G_{i}( \gamma \tau)  & = - \zeta( \gamma)^{p \frac{n}{k}}G_{i}( \tau), \\
 H_{i} ( \gamma \tau) & = - \zeta( \gamma)^{\frac{n}{k}}H_{i}( \tau),
\end{align*}
where $\zeta( \gamma)$ is a 12th root of unity, which depends on $\gamma$ (see  Proposition \ref{trans2}).
\end{prop}
\begin{proof}
Applying the Proposition \ref{transEFH}, we find that for all $ 0 \le i \le k-1$
 \begin{center}
     $G_{i}( \gamma \tau) = \zeta(\gamma)^{p \frac{n}{k}}(-1)^{a_5} e^{\pi i a_{6}}G_{i}(\tau) 
      $
 \end{center}
 where $a_{5}$, $a_6$ are  constants defined as 
 \begin{align}
     a_5 & = 1 + \alpha + \alpha^i,
     a_6 &= -\alpha^{i} \frac{2a +1 }{p^2} \frac{1 - \alpha^{p-1}}{1 - \alpha} + \frac{2a-1}{p^2} \frac{1 - \alpha^{p-1}}{1 - \alpha} \displaystyle \sum_{t=0}^{n/k-2} \alpha^{i + tk +k}
 \end{align}
First, we show that $a_6$ is an integer. We note that 
$$ a_6 = \frac{1 - \alpha^{p-1}}{p(1 -\alpha)} \left( (2a -1)  \frac{\alpha^{i+k}}{p} \frac{1 - \alpha^{n-k}}{1 - \alpha^k}  - (2a +1) \frac{\alpha^i}{p}\right) $$
and since $ \frac{1 - \alpha^{p-1}}{p(1 -\alpha)} $ by our choice of $\alpha$, $a_6$ is a integer if 

$$ (2a -1) \alpha^{i+k}( 1 - \alpha^{n-k}) - (2a +1) \alpha^i(1 - \alpha^k) \equiv 0 \mod{p}
 $$
 and this follows since $2a \alpha^k \equiv 1 \mod{p}$ and $\alpha^n \equiv -1 \mod{p}$. 
The parity of $a_{6}$ is
$$ a_6 \equiv (1 + (p-2) \alpha) ( \alpha^{i+k}( 1 + ( \frac{n}{k} -2) \alpha) + \alpha^i (1 + \alpha)  \equiv  \alpha + \alpha^i \mod{2}. $$
Therefore, 
$$ (-1)^{a_5} e^{- \pi i a_6} = -1$$
and the claim for $G_i(\tau)$ follows. The case of $H_{i}$ is similar. From the definition of $H_{i}$ and the Proposition \ref{transEFH} we find 
\begin{center}
    $H_{i}( \gamma \tau) = \zeta( \gamma)^{\frac{n}{k}} (-1)^{a_{7}}e^{- \pi i a_{8}}H_{i}(\tau)$
\end{center}
where the constants $a_{7}$, $a_{8}$ are defined as  
\begin{align*}
   & a_{7} = \alpha^{i} + \displaystyle \sum_{t=1}^{n/k -1} \alpha^{i + tk -k}, \\
  & a_{8} = - \frac{\alpha^{n +2i -k}}{p} + \displaystyle \sum_{t=1}^{n/k -1} \frac{\alpha^{2(i+tk) -k}}{p}.
\end{align*}
We first show that $a_{8}$ is an integer. As 
 $$ a_8 = \frac{-1}{p}\left( \alpha^{n-k + 2i} - \alpha^{2i -k} \left( \frac{\alpha^{2k} (1 - \alpha^{2n}}{ 1- \alpha^{2k}} -1 \right) \right) $$
it suffices to show that 
$$\alpha^{2n + 2i }(1 - \alpha^{2k}) - \alpha^{2i + 2k} (1 - \alpha^{2n}) + \alpha^{2i} ( 1- \alpha^{2k}),$$
This follows from the fact that $\alpha^n \equiv -1 \mod{p}$. 

Moreover, its parity is 
$$ a_8 \equiv \alpha^{n -k + 2i} + \alpha \left(  \frac{n}{k} -1\right) \equiv \alpha \frac{n}{k} \mod{2}. $$
The parity of $a_7 \in \Z$ is 
$$ a_7 \equiv 1 + \alpha^i + \alpha^i \left((1 + \alpha \left( \frac{n}{k} -2 \right) \right) \equiv 1 + \alpha \frac{n}{k} \mod{2}.$$
Therefore, 
$$ (-1)^{a_7} e^{- \pi i a_8 = -1}$$
and this completes our proof. 
\end{proof}

\begin{proof}[Proof of Proposition \ref{XHmainthm}] To prove that such a product of Siegel functions is a modular unit for $\Gamma_{H}(p)$ we begin by showing that it is a modular unit for $\Gamma_{1}(p)$. More precisely,  by Proposition \ref{X1pcongs} $f(\tau)$ is a modular unit on $X_1(p)$ if and only if 

\begin{align}
    & \displaystyle \sum_{i=0}^{k-1} g_{i} \sum_{t=0}^{n/k -1} \alpha^{2i + 2tk} \equiv 0 \mod{p}, \label{eqq1} \\
        & \displaystyle \sum_{i=0}^{k-1} h_{i} \sum_{t=0}^{n/k -1} \alpha^{2i + 2tk} \equiv 0 \mod{p}, \label{eqq2} \\
     & \displaystyle \sum_{i=0}^{k-1} p \frac{n}{k}g_{i} + \frac{n}{k} h_{i} \equiv 0 \mod{12}.  \label{eqq3} 
        \end{align}
Note that \eqref{eqq1}  and \eqref{eqq2} hold for all $g_i, h_i$ by our choice of $\alpha$.  Suppose \eqref{cong1XH}, \eqref{cong2XH2} hold. Then, since \eqref{eqq3}  and \eqref{cong1XH} coincide, $f(\tau)$    is a modular unit on $X_1(p)$, and it suffices to show that  $f(\gamma \tau ) = f( \tau)$ for $\gamma$ as defined in Lemma  \ref{cosetH}. By the Proposition \ref{transGH} 
\begin{center}
    $f ( \gamma \tau) = (-1)^{\sum_{i=0}^{k-1} g_{i} + h_{i}} \zeta( \gamma)^{\sum_{i=0}^{k-1} \frac{n}{k} (pg_{i} + h_{i}} f (\tau)$
\end{center}
and the result follows from the \eqref{cong1XH} since $\zeta(\gamma)$ is a 12th root of unity\footnote{As argued in \cite[Chapter 3. page 75]{KL}, $\zeta( g) = e^{2p \pi i/12}$.}.

On the contrary, suppose that such a product is a modular unit on $X_{H}(p)$. Then \eqref{cong1XH} holds since $f(\tau)$ is a modular unit on $X_1(p)$ and thus \eqref{eqq3} holds. Moreover, using  Proposition \ref{transGH}, we note that 
$$ f(\gamma\tau) = (-1)^{\sum_{i=0}^{k-1} (g_i + h_i)}f(\tau)$$
where $\gamma \in \Gamma_H(p)$, as define in Lemma \ref{cosetH}. Hence \eqref{cong2XH2} is necessary for $f(\gamma \tau) = f(\tau)$. 
\end{proof}

\section{Basis of Modular Units} \label{basis}
In this section we prove that up to scaling, all modular units on $X_{H} \left( p \right)$ are products of $G_i(\tau)$ and $H_j(\tau)$, with integer exponents satisfying the congruence conditions of Theorem \ref{XHmainthm}. 
\begin{thm} \label{genformofmuXH}
 Let $f$ be a modular unit on $X_{H} \left(p\right)$. Then
 \begin{center}
     $f \left( \tau \right) = C \displaystyle \prod_{i=0}^{k-1} G_{i} \left( \tau \right)^{g_{i}} H_{i} \left( \tau \right)^{h_{i}}$
 \end{center}
 for some constant $C$ and integer exponents $g_{i}$ and $h_{j}$ satisfying:
  \begin{itemize}
     \item[(i)] $ \displaystyle \frac{n}{k} \left( \sum_{i=0}^{k-1} p g_{i} + h_{i}) \right) \ \equiv 0 \mod{12},$
     \item[(ii)] $ \displaystyle \sum_{i=0}^{k-1} g_{i} + h_{i} \ \equiv \ 0 \mod{2}$.
 \end{itemize}
 \end{thm}

This result is in large part a consequence of the following proposition. 
\begin{prop} \label{expsH}
 Let $f(\tau) = \displaystyle \prod_{i=0}^{n-1}E_{i}(\tau)^{e_{i}}F_{i}(\tau)^{f_{i}}$ be a modular unit on $X_{1}(p)$, with integer exponents $e_{i}$ and $f_{j}$ satisfying the congruence conditions of Proposition \ref{X1pcongs}. Suppose that $f(\gamma \tau) = f(\tau)$ for $\gamma$ as defined in Lemma \ref{cosetH}. Then $e_{i} = e_{i + t k}$ and $ f_{i} =f_{i + tk } $ for all $0 \le i \le k-1$ and $ 0 \le t \le \frac{n}{k} -1$.
\end{prop}
\begin{proof}
 By Proposition \ref{transEFH} we note that
     $ f( \gamma \tau) = C \Tilde{f}(\tau)$,
 for some constant $C \in \mathbb{C}^{*}$ and $\Tilde{f} ( \tau) = \displaystyle \prod_{i=0}^{n-1} E_{i}( \tau)^{e'_{i}} F_{i}(\tau)^{f'_{i}}$ where the exponents $e'_{i}, f'_{i}$ are defined as 
 \begin{center}
   $  e'_{i} = \left\{
	\begin{array}{ll}
		e_{i+n-k}  & \mbox{if } 0 \le i \le k-1, \\
		e_{i-k} & \mbox{if } k \le i \le n-1,
	\end{array}
\right.  \mbox{and} \ \ 
    f'_{i} = \left\{
	\begin{array}{ll}
		f_{i + k}  & \mbox{if } 0 \le i \le n -1 -k, \\
		f_{i- n + k} & \mbox{if } n - k \le i \le n-1,
	\end{array}
\right.$
\end{center}
and we note that $\Tilde{f}(\tau)$ is a modular unit on $X_{1}(p)$ since 
\begin{align*}
    \displaystyle \sum_{i=0}^{n-1} e'_{i} \alpha^{2i+2}  &  = \displaystyle \sum_{i=0}^{k-1} e_{i+n -k} \alpha^{2i} + \sum_{i=k}^{n-1} e_{i-k} \alpha^{2i} \equiv   \displaystyle \sum_{i=0}^{k-1} e_{i+n -k} \alpha^{2i + 2n} + \sum_{i=k}^{n-1} e_{i-k} \alpha^{2i + 2n}  \equiv   \displaystyle \sum_{j=n-k}^{n-1} e_{j} \alpha^{2j + 2k} + \sum_{j=0}^{n-1 -k} e_{j} \alpha^{2j + 2k } \\ & \equiv \alpha^{2k} \displaystyle \sum_{j=0}^{n-2} e_{j} \alpha^{2j} \equiv 0 \mod{p}; \\
     \displaystyle \sum_{i=0}^{n-1} f'_{i} \alpha^{2i+2}  &  = \displaystyle \sum_{i=0}^{n - 1 + k} f_{i+k} \alpha^{2i} +  \sum_{i=n-k}^{n-1} f_{i-n + k} \alpha^{2i} \equiv - \displaystyle \sum_{i=0}^{n-k-1} f_{i+k} \alpha^{2i} \alpha^{n} - \sum_{i=n-k}^{n-1} f_{i-n+k} \alpha^{2i} \alpha^n  \\ & \equiv -  \displaystyle \sum_{j=k}^{n-1} f_{j} \alpha^{2(j-k)} \alpha^{n} - \sum_{j=0}^{k-1} f_{j} \alpha^{2(j+n-k} \alpha^{n}  \equiv - \alpha^{n -2k} \displaystyle \sum_{j=0}^{n-1} f_{j} \alpha^{2j} \equiv 0 \mod{p}; \\
     \displaystyle \sum_{j=0}^{n-1} p e_{j}' + f'_{j} & =  \sum_{j=0}^{n-1} p e_{j} + f_{j} \equiv 0 \mod{12}. 
     \end{align*}
Applying Proposition \ref{EFtoAB} to both $f$ and $\Tilde{f} $ we find 
\begin{align*}
   & f(\tau) = \left( \displaystyle \prod_{i=0}^{n-2} A_{i}( \tau)^{e_{i}} \right)A_{n-1}( \tau)^{a(f)} \left( \displaystyle \prod_{i=0}^{n-2} B_{i}(\tau)^{f_{i}} \right)B_{n-1}(\tau)^{b(f)}; \\
      & \Tilde{f}(\tau) = \left( \displaystyle \prod_{i=0}^{n-2} A_{i}( \tau)^{e'_{i}} \right)A_{n-1}( \tau)^{a(\Tilde{f})} \left( \displaystyle \prod_{i=0}^{n-2} B_{i}(\tau)^{f'_{i}} \right)B_{n-1}(\tau)^{b(\Tilde{f})}; \\
      \end{align*}
with $a(g)$ and $b(g)$ as previously defined (see Proposition \ref{EFtoAB}). As $f(\gamma \tau) = f ( \tau)$, it follows that  $f( \tau )/ \Tilde{f}(\tau)$ is a non-zero constant function on $X_{1}(p)$, that is, 
\begin{center}
   $  \left( \displaystyle \prod_{i=0}^{n-2} A_{i}( \tau)^{e_{i} - e'_{i}} \right)A_{n-1}( \tau)^{a(f) - a(\Tilde{f})} \left( \displaystyle \prod_{i=0}^{n-2} B_{i}(\tau)^{f_{i} - f'_{i}} \right)B_{n-1}(\tau)^{b(f) - b( \Tilde{f})} $ 
   \end{center}
   is constant. We apply Proposition \ref{constAB} to the above to deduce that there exists an integer $a$ such that 
 \begin{align}
& e_{i} - e_{i+n-k} = a \ \text{for} \ 0 \le i \le k-1, \label{eq1}\\
& e_{i} - e_{i-k} = a \ \text{for} \ k \le i \le n-2,  \label{eq2} \\
& a(f)  - a(\Tilde{f}) = an_{1}, \label{eq3}  \\
& f_{i } - f_{i+k} = a \ \text{for} \ 0 \le i \le n-1-k,  \label{eq5} \\
& f_{i} - f_{i+k-n} = a \ \text{for} \ n + 1-k \le i \le n-2 \label{eq7}, \\
& b(f)- b(\Tilde{f}) = an_{2},  \label{eq8}
\end{align}
where $n_{1}, n_{2}$ are as previously defined (see Proposition \ref{constAB}). From \eqref{eq2}, taking $i =k $ we obtain 
\begin{center}
    $e_{k} = a + e_{0}$
\end{center}
and further setting $i = tk$ for $1 \le t \le \frac{n}{k} -1 $ in \eqref{eq2}, inductively we obtain 
\begin{center}
    $e_{n-k} = ( \frac{n}{k} -1)a + e_{0} $.
\end{center}
It follows from \eqref{eq1} with $i =0$ that 
\begin{center}
    $e_{0} = a + e_{n-k}$
\end{center}
and combining the previous two expressions we conclude that $a =0$. The result now follows. 
\end{proof}

\begin{proof}[Proof of Theorem \ref{genformofmuXH}] Any modular unit $f$ on $X_{H}(p)$ is a modular unit on $X_{1}(p)$ satisfying $f(\gamma \tau) = f ( \tau)$; and hence it is a product of $G_{i}$ and $H_{j}$ by Proposition \ref{expsH}. The exponents necessarily satisfy the stated congruences by Proposition \ref{XHmainthm}. 
\end{proof}

The criterion of Theorem \ref{genformofmuXH} can be used to construct an explicit basis for the group of modular units on $X_{H}(p)$, modulo constants; or equivalently, we compute a basis for the group of integer solutions to the stated congruences. 

Let $R = \text{lcm} \left( 2, \frac{12}{\text{gcd} \left(12, \frac{n}{k} \right)} \right)$ and observe that the exponents $g_{i}, h_{j}$ satisfy the congruences of Theorem \ref{genformofmuXH} if and only if
\begin{center}
    $\displaystyle \sum_{i=0}^{k-1} pg_{i} + \displaystyle \sum_{i=1}^{k-1} h_{i} \equiv 0 \mod{R}$.
\end{center}
For $0 \le j \le k-1$, define 
\begin{center}
$C_{j} \left( \tau \right) = G_{j} \left( \tau \right) H_{k-1} \left( \tau \right)^{-p}$    
\end{center}
and further define $D_{j}$ for $0 \le j \le k-1$ by 
\begin{align*} 
    D_{j} \left( \tau \right)  =
    \begin{cases}
H_{j}\left( \tau \right) H_{k-1} \left( \tau \right)^{-1}  \ \text{if} \ 0 \le j \le k-2 \\
H_{k-1}\left( \tau \right)^{R} \ \text{if} \ j = k-1.
    \end{cases}
\end{align*}

Note that the  functions defined above are all modular units on $X_{H}\left( p \right)$ by Proposition \ref{XHmainthm}, and their exponents correspond to a basis for the set of integer solutions of the above congruence. We deduce the following.

\begin{thm} \label{FHpBasis}
The $2k-1$ functions $C_{0}, \ldots, C_{k-1}, D_{1}, \ldots, D_{k-1}$ are a basis for the group of modular units on $X_{H}\left( p\right)$ modulo constants.
\end{thm}

\begin{remark}
 Above we also defined $D_{0}$, but this does not appear in our basis. A similar argument to that of Proposition \ref{constAB} shows that a certain product of all $C_{i}$ and $D_{j}$ is a constant function.   
\end{remark}

\section{The Cuspidal Class Number} \label{cuspidalclassnumber}
The \textit{cuspidal group} of $X_{H}(p)$ is the subgroup of the Jacobian generated by differences of cusps. It is simply the  quotient group
\begin{center}
   $\mathcal{C}_{H}(p) = \Divv_{\cusps}^{0}(X_{H}(p))/\langle \divv (f) : f \in  \mathcal{F}_{H}(p) \rangle$ 
\end{center}
where $\Divv_{\cusps}^{0}(X_{H}(p))$ is the group of degree $0$ divisors supported on cusps and $\mathcal{F}_H(p)$ is the group of modular units. This is a finite group, due to the theorems of Manin \cite{Manin} and Drinfeld \cite{Drinfeld}, and its order, known as the \textit{cuspidal class number}, is denoted by $h_{H}(p)$. In this section, we give an explicit formula for $h_{H}(p)$ and prove Theorem \ref{thm1}.

The group $\Divv_{\cusps}^{0}(X_{H}(p))$ is generated by differences of cusps, and hence $\Divv_{\cusps}^{0}(X_{H}(p)) \cong \Z^{2k-1}$, since $X_H(p)$ has $2k$ cusps in total, see Proposition \ref{cuspsprop} and the representatives given below. Therefore, computing $h_H(p)$ reduces to computing the determinant of the matrix whose entries are determined by the divisors of the modular units in the explicit basis of $\mathcal{F}_H(p)/\mathbb{C}^*$ given in Theorem \ref{FHpBasis}.

\subsection{Orders of Vanishing} \label{ordersofvanishing}
To understand the  group of relations in the above quotient, we begin by fixing an explicit set of representatives for the cusps. By Proposition \ref{cuspsprop}, we deduce that $X_{H}(p)$ has $2k$ cusps in total, with explicit representatives 
\begin{itemize}
    \item $P_i = [ \frac{\alpha^i}{p}]$;
    \item $Q_i = [\frac{p}{\alpha^i}]$;
\end{itemize}
for $0 \le i \le k-1$, where $[x]$ denotes the class of $x \in \mathbb{P}^1(\Q)$ under the action of $\Gamma_H(N)$.

Note that $X_0(p)$ has two cusps, denoted by $c_{\infty}$ (represented by $\infty$) and $c_0$ (represented by $0$). The cusps $P_i$ (resp. $Q_i$) map to $c_\infty$ (resp. $c_0$), under the quotient map $X_H(p) \rightarrow X_0(p)$.

For the rest of this subsection we write  $(a_{1}, a_{2})_{1} := a_{1}$ and $\{ x \} = x - \floor{x}$ for the fractional part of $x \in \R$. To compute the order of vanishing of the previously constructed modular units at any cusp we use the following elementary result. 

\begin{lem} \label{ord1}
Let  $g (\tau ) = \displaystyle \prod_{\mathbf{a} \in I} g_{\mathbf{a}}(\tau )^{b_{\mathbf{a}}}$ be  a modular unit on $X_{H}(p)$. Then 
\begin{center}
    $ \divv (g) =  \displaystyle \sum_{c} \frac{w_{c}}{2} \sum_{\mathbf{a} \in I} B_{2}(\{ (\mathbf{a}\gamma_{c})_{1}\})$ 
\end{center}
where the first sum is over the cusps of $X_{H}(p)$, $w_{c}$ is the width of the cusp $c$ and $\gamma_{c} \in \SL_{2}(\Z)$ is such that $\gamma_{c} \cdot i \infty = c$.  
\end{lem}
\begin{proof}
This is a consequence of Propositions \ref{trans2} and \ref{ordga}.   
\end{proof}

\begin{prop} \label{ordvanXhp}
Let $P_{i}$ and $Q_{i}$ be the cusps of $X_{H} \left( p \right)$ described above. Then, for all $0 \le i \le k-1$ and  $ 0 \le j \le k-1$ we have
\begin{itemize}
    \item[(i)] $ \text{ord}_{P_{i}} \left( C_{j} \right) = \displaystyle \sum_{t=0}^{n/k-1} \frac{p}{2} B_{2} \left( \left \{ \frac{\alpha^{j +tk +i}}{p} \right \} \right) - \frac{pn}{12k}, $\\ 
    \item[(ii)] $  \text{ord}_{Q_{i}} \left( C_{j} \right) = \frac{n}{12k} - \displaystyle \frac{p^{2}}{2} \sum_{t=0}^{n/k-1} B_{2} \left( \left \{ \frac{\alpha^{k-1+kt +i}}{p} \right \} \right). $
\end{itemize}
Similarly, for all $0 \le i \le k-1$ and $0 \le j \le k-2$ we have
\begin{itemize}
    \item[(iii)] $\text{ord}_{P_{i}} \left( D_{j} \right) = 0,$\\
    \item[(iv)] $\text{ord}_{Q_{i}} \left( D_{j} \right) = \displaystyle \sum_{t=0}^{n/k-1} \frac{p}{2} B_{2} \left( \left \{ \frac{\alpha^{j +tk +i}}{p} \right \} \right) - \sum_{t=0}^{n/k-1} \frac{p}{2} B_{2} \left( \left \{ \frac{\alpha^{k-1 +tk +i}}{p} \right \} \right),  $
\end{itemize}
and when $j=k-1$ we have 
\begin{itemize}
    \item[(v)] $ \text{ord}_{P_{i}} \left( D_{k-1} \right) = \frac{nR}{12k},$\\ 
\item[(vi)] $ \text{ord}_{Q_{i}} \left( D_{k-1} \right) = \displaystyle \sum_{t=0}^{n/k-1} \frac{Rp}{2} B_{2} \left( \left \{ \frac{\alpha^{k-1 +tk +i}}{p} \right \} \right).$
\end{itemize} 
\end{prop}
\begin{proof}
The cusps $P_{i}$ all have width $1$, whilst all $Q_{j}$ have width $p$. 
For any $0 \le i \le k-1$, there exist $a_{i}, b_{i}, c_{i}, d_{i} \in \Z $ such that 
\begin{center}
    $ \gamma_{i} = \begin{psmallmatrix}
        \alpha^{i} & a_{i} \\
        p & b_{i}
    \end{psmallmatrix} ,  \gamma'_{i} = \begin{psmallmatrix}
        p & c_{i} \\
        \alpha^{i} & d_{i}
    \end{psmallmatrix} \in \SL_{2}(\Z) $.
\end{center}
Then $[\gamma_{i} i \infty] = P_{i}$ and $ [\gamma_{i}' i \infty] = Q_{i}$; and the above formulae follow by Lemma \ref{ord1}.
\end{proof}

\subsection{Circulant Matrices} \label{circulant}
In our computation of $h_H(p)$ we encounter a specific circulant matrix. We briefly discuss such matrices and their determinants. 

Let $m \in \Z_{>0}$ and $\mathbf{c} = (c_{1}, \ldots, c_{m}) \in \mathbb{C}^{m}$. The \textit{circulant matrix} defined by $\mathbf{c}$ is the $m \times m $ matrix
\begin{center}
    $\Circ(\mathbf{c}) := \begin{bmatrix}
        c_{1} & c_{2} & \ldots & c_{m-1} & c_{m} \\
        c_{2} & c_{3} & \ldots & c_{m}  & c_{1} \\
        \vdots & \vdots & \ddots & \vdots & \vdots \\
        c_{m} & c_{1} & \ldots & c_{m-2} & c_{m-1}
    \end{bmatrix}
    $
\end{center}
The following result is well known, see for instance \cite{ingleton1956rank}. 

\begin{prop} \label{detcirc}
Let $C =\Circ((c_{1}, \ldots, c_{m}))$ be an $m \times m $ circulant matrix. The determinant of $C$ is 
\begin{center}
    $\det \left( C\right) = \displaystyle \prod_{j=1}^{m} \left( c_{1} + c_{2} \zeta^{j} + c_{3} \zeta^{2j} + \ldots + c_{m} \zeta^{j(m-1)} \right)$
\end{center}
where $\zeta$ is a primitive $m$th root of unity. 
\end{prop}

\subsection{The Cuspidal Class Number Formula: Proof of Theorem \ref{thm1}}
\label{formula1}

Let $\chi$ be a Dirichlet character modulo $p$ . The \textit{generalised Bernoulli number} corresponding to $ \chi$ is defined as 
\begin{center}
    $ B_{2, \chi} = p \displaystyle \sum_{a=1}^{p} \chi(a) B_{2} \left( \frac{a}{p} \right)$
\end{center}
where $B_{2}(x) = x^{2} - x + \frac{1}{6}$ is the second Bernoulli polynomial. 

For $0 \le i \le k-1$, we set: 
\begin{itemize}
\item $ \Tilde{a_{i}} = \displaystyle\frac{p}{2} \sum_{t=0}^{n/k-1} B_{2} \left( \left \{ \frac{\alpha^{i+tk}}{p} \right \} \right);$
\item $ a_{i} = \Tilde{a_{i}} - \frac{n}{12k}. $
\end{itemize} 
We begin with two elementary results. 
\begin{lem} \label{columnsum}
 $ \displaystyle \sum_{i=0}^{n-1} \Tilde{a}_{i} =  -\frac{n}{12}$.
\end{lem}
\begin{proof}
    Observe that $B_{2} \left( \{ -x \} \right)  = B_{2} \left( \{ x \} \right)$ and hence, by the definition of $\alpha$ our sum is  
\begin{align*}
\displaystyle \sum_{i=0}^{n-1} B_{2} \left( \left \{ \frac{\alpha^{i}}{p} \right \} \right)  & = \displaystyle \sum_{x=1}^{n} B_{2} \left( \frac{x}{p} \right) = \frac{1}{6p^{2}}n(n+1)(2n+1) - \frac{1}{2p}n(n+1) + \frac{1}{6}n = \frac{-n}{6p}.
\end{align*}
\end{proof}

\begin{lem} \label{detB'}
With notation as above, 
\begin{center}
   $\det( \Circ( a_0, \ldots, a_{k-1})) = -\frac{n}{6} \displaystyle \prod_{\chi \in \hat{G}_H, \chi \ne 1 } \frac{1}{4}B_{2,\chi}.$ 
\end{center}
\end{lem}
\begin{proof}
Let $\zeta$ be primitive $k$th root of unity. Applying Lemmata \ref{detcirc}  and \ref{columnsum} we get 
\begin{center}
   $\det(B') = \displaystyle \prod_{i=0}^{k-1}( a_{0} + \zeta^{j}a_{1} + \ldots \zeta^{j(k-1)}a_{k-1}) = - \frac{n}{6} \displaystyle  \prod_{i=1}^{k-1} ( \Tilde{a}_{0} + \zeta^{j}\Tilde{a}_{1} + \ldots + \zeta^{j(k-1)} \Tilde{a}_{k-1})$.  
\end{center}
For any $j \ne 0$, there exists a unique non-principal $\chi_j \in \hat{G}_H$ such that $\chi_j(\alpha) = \zeta^j$.
It follows that 
\begin{center}
    $ \displaystyle \sum_{i=0}^{k-1} \chi_j (\alpha^{i}) \Tilde{a}_{i} = \displaystyle \sum_{i=0}^{k-1} \chi_j ( \alpha^{i}) \sum_{t=0}^{n/k-1} \frac{p}{2} B_{2} \left( \{ \frac{\alpha^{i+tk}}{p} \} \right) =  \frac{p}{ 2} \displaystyle \sum_{i=0}^{k-1}  \sum_{t=0}^{n/k-1} \chi_j ( \alpha^{i + tk}) B_{2} \left( \{ \frac{\alpha^{i+tk}}{p} \} \right) = \frac{p}{2} \displaystyle \sum_{a=1}^{n} \chi_j(a) B_{2}( \frac{a}{p})$
\end{center}
 where the last equality follows from the definition of $\alpha$. Moreover, since $B_{2}(x) = B_{2}(1-x)$, the above is in fact equal to $\frac{1}{4}B_{2, \chi}$. 
Notably, sending $\chi_j \mapsto \zeta_j$ is a bijection from $\hat{G}_H \setminus \{ 1 \}$ to $\{ \zeta^j  \ \vert \ 1 \le j \le k-1 \}$ and our claim follows. 

\end{proof}

\begin{proof}[Proof of Theorem \ref{thm1}]
Recall that $h_{H}(p) = \det(C)$ where $C$ is a matrix representing the divisors of the explicitly constructed basis of modular units in Theorem \ref{FHpBasis}.

We fix 
$$ P_{0} - Q_{k-1}, \ldots, P_{k-1} - Q_{k-1}, Q_{0} - Q_{k-1}, \ldots, Q_{k-2} - Q_{k-1}$$
a basis for $\Divv_{\text{cusps}}^0(X_H(p))$. Then expressing the divisors of $C_{0}, \dots, C_{k-1}, D_{1}, \dots D_{k-1}$ in the above basis, we find that $C = AB$ where 
  \begin{center}
   $ A =  \begin{psmallmatrix}
        1 & 0 & \ldots & 0 & 0 &  \ldots & 0 & -p \\
        0 & 1 & \ldots & 0 & 0 &  \ldots & 0 & -p \\
        \vdots & \vdots & \ddots & \vdots & \vdots & \ddots & \vdots & \vdots \\
        0 & 0 & \ldots  & 1 & 0  & \ldots & 0 & -p \\
        0 & 0 & \ldots & 0 & 1 &  \ldots & 0 & -1 \\
        \vdots & \vdots & \ddots &  \vdots & \vdots & \ddots & \vdots & \vdots \\
        0 & 0 & \ldots & 0 & 0 & \ldots & 1 & -1 \\
        0 & 0 & \ldots & 0 & 0 &  \ldots & 0 & R
    \end{psmallmatrix} $
\end{center}
\begin{center}
    $ B = \begin{psmallmatrix}
        \Tilde{a}_{0} & \Tilde{a}_{1} & \ldots & \Tilde{a}_{k-1} & \frac{n}{12k} & \frac{n}{12k} & \ldots  & \frac{n}{12k} & \frac{n}{12k} \\
        \Tilde{a}_{1} & \Tilde{a}_{2} & \ldots & \Tilde{a}_{0} & \frac{n}{12k} & \frac{n}{12k} & \ldots  & \frac{n}{12k} & \frac{n}{12k} \\
        \vdots & \vdots & \ddots & \vdots & \vdots & \vdots & \ddots & \vdots & \vdots \\
        \Tilde{a}_{k-1} & \Tilde{a}_{0} & \ldots & \Tilde{a}_{k-2} & \frac{n}{12k} & \frac{n}{12k} & \ldots & \frac{n}{12k} & \frac{n}{12k} \\
        \frac{n}{12k} & \frac{n}{12k} & \ldots & \frac{n}{12k} & \Tilde{a}_{1} & \Tilde{a}_{2} & \ldots &  \Tilde{a}_{k-2} & \Tilde{a}_{k-1} \\
        \frac{n}{12k} & \frac{n}{12k} & \ldots & \frac{n}{12k} & \Tilde{a}_{2} & \Tilde{a}_{3} & \ldots &  \Tilde{a}_{k-1} & \Tilde{a}_{0} \\
        \vdots & \vdots & \ddots & \vdots & \vdots & \vdots & \ddots & \vdots & \vdots \\
        \frac{n}{12k} & \frac{n}{12k} & \ldots & \frac{n}{12k} & \Tilde{a}_{k-1} & \Tilde{a}_{0} & \ldots & \Tilde{a}_{k-4} & \Tilde{a}_{k-3} 
        \end{psmallmatrix}  $
\end{center}
where for $ 1 \le i \le k$, the $ith$ row of $C$ represents the divisor of $C_{i-1}$ and for $k + 1 \le i \le 2k-1$ the $ith$ row represents the divisor of $D_{i-k}$. As $\det(A) = R$, the problem reduces to computing the determinant of $B$. To do this, we consider the $2k \times 2k$ matrix
\begin{center}
    $ J = \begin{psmallmatrix}
        a_{0} & a_{1} & \ldots & a_{k-1} & 0 &  0  & \ldots & 0 \\
        a_{1}  & a_{2} & \ldots & a_{0} & 0 & 0 & \ldots & 0 \\
        \vdots & \vdots & \ddots & \vdots & \vdots & \vdots & \ddots & \vdots \\
        a_{k-1}  & a_{0}  & \ldots & a_{k-2}  & 0 &  0 & \ldots & 0 \\
        0 & 0  & \ldots & 0 & a_{0}  & a_{1} & \ldots & a_{k-1} \\
      0  & 0 & \ldots & 0  & a_{1} & a_{2}  & \ldots & a_{0}  \\
        \vdots & \vdots & \ddots & \vdots & \vdots & \vdots & \ddots & \vdots \\
        0 & 0  & \ldots & 0 & a_{k-1}  & a_{0}  & \ldots & a_{k-2} 
        \end{psmallmatrix} $.
\end{center}
Note that this is a block matrix, and from Lemma \ref{detB'}, its determinant is 
\begin{center}
  $\det(J) =  \left( \frac{n}{6} \displaystyle \prod_{\chi \in \hat{G}_H, \chi \ne 1} \frac{1}{4}B_{2,\chi} \right)^{2}.$  
\end{center}
Equivalently, we compute the determinant of $J$ by applying the following row and column operations:
\begin{itemize}
    \item[1.] add all rows to the $(k+1)$th row;
    \item[2.] add all columns to the last column;
    \item[3.] factor out $2k$ from the $(k+1)$ the row;
    \item[4.] subtract the $(k+1)$th row from all others. 
\end{itemize}
Applying Lemma \ref{columnsum} repeatedly, we obtain the following matrix:
\begin{center}
    $ 2k \begin{psmallmatrix}
        \Tilde{a}_{0} & \Tilde{a}_{1} & \ldots & \Tilde{a}_{k-1} & \frac{n}{12k} & \frac{n}{12k} & \ldots & 0\\
        \Tilde{a}_{1} & \Tilde{a}_{2} & \ldots & \Tilde{a}_{0} & \frac{n}{12k} & \frac{n}{12k} & \ldots & 0 \\
        \vdots & \vdots & \ddots & \vdots & \vdots & \vdots & \ddots & \vdots \\
        \Tilde{a}_{k-1} & \Tilde{a}_{0} & \ldots & \Tilde{a}_{k-2} & \frac{n}{12k} & \frac{n}{12k} & \ldots & 0 \\
        -\frac{n}{12k} & -\frac{n}{12k} & \ldots & -\frac{n}{12k} & - \frac{n}{12k} & - \frac{n}{12k}  & \ldots & - \frac{2n}{12} \\
        \frac{n}{12k} & \frac{n}{12k} & \ldots & \frac{n}{12k} & \Tilde{a}_{1} & \Tilde{a}_{2} & \ldots & 0 \\
        \vdots & \vdots & \ddots & \vdots & \vdots & \vdots & \ddots & \vdots \\
        \frac{n}{12k} & \frac{n}{12k} & \ldots & \frac{n}{12k} & \Tilde{a}_{k-1} & \Tilde{a}_{0} & \ldots & 0          
        \end{psmallmatrix}. $
\end{center}
Expanding from the last column we obtain 
\begin{center}
    $\text{det} \left( J \right) = - \frac{kn}{3} \text{det} \left( B \right)$.
\end{center}
Therefore, the determinant of $B$ is 
\begin{center}
    $\det(B) = - \frac{3}{nk} \left( \frac{n}{6} \displaystyle \prod_{\chi \in \hat{G}_H, \chi \ne 1} \frac{1}{4}B_{2,\chi} \right)^{2} $
\end{center}
and recalling that $R = \lcm(2, \frac{12}{\gcd(12, \frac{n}{k}}$ and $n = \frac{N-1}{2}$, the result follows. 
\end{proof}

\section{Rational Subgroups of Modular Units and Rational Cuspidal Class Numbers} \label{subgroups}
Given the explicit nature of the cuspidal class number $h_{H}(p)$ computed in the previous section, it is natural to ask whether a similar formula exists for the order of the rational cuspidal group $\mathcal{C}_{H}(p)(\Q)$. One obstacle in this problem is that,  apriori, there is no clear presentation of $\mathcal{C}_{H}(p)(\Q)$ as a quotient group using modular units. To circumvent this issue, we study subgroups of $\mathcal{C}_{H}(p)(\Q)$ for which such a presentation exists, we compute the groups of modular units corresponding to such groups and then relate them to $\mathcal{C}_{H}(p)(\Q)$.

Throughout this paper, we work with a canonical model of $X_H(p)$ over $\Q$, and with the cusps $P_{0}, \ldots, P_{k-1}$ defined over $\Q$, whilst $Q_{0}, \ldots, Q_{k-1}$ are defined over $(\Q( \zeta_{p})^{+})^{H}$ and form a single Galois orbit. Here, we identify $H$ with a subgroup of the Galois group of the extension $\Q \subseteq \Q(\zeta_p)^+$. 

\subsection{The Subgroup \texorpdfstring{$\mathcal{C}_{H}^{\infty}(p)$}{a}} \label{h_Hinf}
A subgroup of $\mathcal{C}_{H}(p)(\Q)$ is the subgroup generated by the rational cusps $P_{0}, \ldots, P_{k-1}$,  which we denote by $\mathcal{C}_{H}^{\infty}(p)$ from now on. Moreover, the modular units on $X_{H}(p)$ with divisors supported on $P_{0}, \ldots, P_{k-1}$ form a  subgroup of the group of modular units, which we denote by $\mathcal{F}_{H}^{\infty}$, and  thus we have 
\begin{center}
    $\mathcal{C}_{H}^{\infty}(p) = \Divv_{\infty-\cusps}^{0}(X_{H}(p))/ \langle \divv (f)  \  \vert \ f \in \mathcal{F}_{H}^{\infty}(p) \rangle $
\end{center}
where  $\Divv_{\infty-\text{cusps}}^{0}(X_{H}(p))$ is the group of degree $0$ divisors on $X_{H}(p)$ supported on $P_{0}, \ldots, P_{k-1}$. 

To determine the order of $\mathcal{C}_H^\infty(p)$, denoted by $h_H^\infty(p)$, we give a basis for $\mathcal{F}_H^\infty(p)$, modulo constants.

\begin{thm} \label{FHinfcriterion}
Let $f( \tau) = \displaystyle \prod_{i=0}^{k-1}G_{i}(\tau)^{g_{i}}$ with exponents $g_{i} \in \Z$. Then $f$ is an element of $\mathcal{F}_{H}^{\infty}(p)$ if and only if the exponents satisfy 
\begin{center}
    $\displaystyle \sum_{i=0}^{k-1} g_{i} =0.$
\end{center}
Moreover, up to scaling, all modular units on $X_{H}(p)$ whose divisors are supported on cusps $P_0, \ldots, P_{k-1}$ are of this form. 
\end{thm}

\begin{proof}
 Suppose that $f$ is as above and that the sum of the exponents is $0$. By Proposition \ref{XHmainthm}, $f$ is a modular unit on $X_{H}(p)$. Moreover, by Proposition \ref{ordga}, the order of vanishing of $f$ at any cusp $Q_{i}$ is 
 \begin{center}
     $\frac{n}{12k} \displaystyle \sum_{i=0}^{k-1}g_{i} = 0 $
 \end{center}
and hence $f \in \mathcal{F}_{H}^{\infty}(p)$. The converse of the first statement is also clear, from the order of vanishing at $Q_{i}$. 

Any element of $\mathcal{F}_{H}^\infty (p)$ is a modular unit on $X_{1}(p)$, with divisor supported on cusps lying above $c_{\infty} \in X_{0}(p)$, and thus by Proposition \ref{F1pinfbasis}, it's necessarily of the form 
$f \left( \tau \right) = \displaystyle \prod_{i=0}^{n-1} E_{i} \left( \tau \right)^{e_{i}}$ with $e_{i} \in \Z $ satisfying: 
\begin{itemize}
    \item $ \displaystyle \sum_{i=0}^{n-1} e_{i} \alpha^{2i} \  \equiv 0 \ \text{mod} \ p;$
    \item $ \displaystyle \sum_{i=0}^{n-1} e_{i} = 0 .$
\end{itemize}
As $f$ is a modular unit on $X_{H}(p)$ and a product of $E_{i}$, we apply Proposition \ref{expsH} to conclude that $e_{i}= e_{i+tk}$ for all $0 \le i \le k-1$ and $0 \le t \le n/k-1$; and hence $f$ is as claimed.
\end{proof}

\begin{cor} \label{FHinf}
The $k-1$ functions $G_{0}(\tau)/G_{k-1}(\tau), G_{1}(\tau)/G_{k-1}(\tau), \ldots, G_{k-2}(\tau)/G_{k-1}(\tau)$ form a basis for the group of modular units  $\mathcal{F}_{H}^{\infty}(p)$, modulo constants.
\end{cor}
\begin{proof}
By Theorem \ref{FHinfcriterion}, all of the above functions are elements of $\mathcal{F}_{H}^{\infty}(p)$, and their exponents are a complete basis for the integer solutions of the necessary conditions.     \end{proof}

On a similar note, we define $\mathcal{C}_H^0(p)$, the subgroup of the cuspidal group generated by the cusps $Q_0, \ldots, Q_{k-1}$. The two sets of cusps $\{Q_i\}$ and $\{P_i \}$ are swapped by the Fricke involution $w_p$, and hence $\mathcal{C}_H^0(p) \cong \mathcal{C}_H^\infty(p)$, and moreover, 
$$ \mathcal{C}_H^0(p) \cong \Divv_{0-\text{cusps}}^0(X_H(p))/ \langle \divv f \ \vert \ f \in \mathcal{F}_H^0(p) \rangle $$
where $\Divv_{0-\text{cusps}}^0(X_H(p))$ is the group of degree $0$ divisors on $X_H(p)$ supported on $Q_0, \ldots, Q_{k-1}$, and $\mathcal{F}_H^0(p)$ is the group of modular units whose divisors are in $\Divv_{0-\text{cusps}}^0(X_H(p))$. Note that 
$$ w_p(\divv (G_i(\tau)/G_{k-1}(\tau)) = \divv (H_i(\tau)/H_{k-1}(\tau))$$
for all $0 \le i \le k-2$, and hence we obtain the following. 
\begin{cor} \label{F_H0basis}
  The $k-1$ functions $H_{0}(\tau)/H_{k-1}(\tau), H_{1}(\tau)/H_{k-1}(\tau), \ldots, H_{k-2}(\tau)/H_{k-1}(\tau)$ form a basis for the group of modular units  $\mathcal{F}_{H}^{0}(p)$, modulo constants.  
\end{cor}

\begin{proof}[Proof of Theorem \ref{thm2}]
 We fix $P_{0} -P_{k-1}, \ldots, P_{k-2} - P_{k-1}$, a basis for the group $\Divv_{\infty-\text{cusps}}^0(X_H(p))$. 

 The divisors of the basis elements constructed  Corollary \ref{FHinf} are rows of the $(k-1) \times (k-1)$ matrix
 \begin{center}
    $D = \begin{psmallmatrix}
    \Tilde{a}_{0} - \Tilde{a}_{k-1} &  \Tilde{a}_{1} - \Tilde{a}_{0} & \ldots  & \Tilde{a}_{k-2} - \Tilde{a}_{k-3} \\
    \Tilde{a}_{1} - \Tilde{a}_{k-1} & \Tilde{a}_{2} - \Tilde{a}_{0} & \ldots &  \Tilde{a}_{k-1} - \Tilde{a}_{k-3} \\
    \vdots & \vdots & \ddots & \vdots \\
    \Tilde{a}_{k-2} - \Tilde{a}_{k-1} & \Tilde{a}_{k-1} - \Tilde{a}_{0} & \ldots & \Tilde{a}_{k-4} - \Tilde{a}_{k-3}
 \end{psmallmatrix}$ 
 \end{center}
 where $\Tilde{a}_{i} = \displaystyle \frac{p}{2} \sum_{t=0}^{s/k-1} B_{2} \left( \left \{ \frac{\alpha^{i+ tk}}{p} \right \} \right)$ and hence 
 
 $$h_{H}^{\infty}(p) = \vert \det D \vert.$$ 

To compute this, we consider the $k \times k$ circulant matrix 
\begin{center}
    $D' = \Circ(\Tilde{a}_{0}, \ldots, \Tilde{a}_{k-1})$.
\end{center}
Applying Lemma \ref{detcirc} we get 
\begin{center}
   $\det D' = ( \displaystyle \sum_{i=0}^{k-1} \Tilde{a}_{i}) \displaystyle \prod_{j=1}^{k-1} ( \Tilde{a}_{0} + \Tilde{a}_{1} \zeta^{j} + \ldots \Tilde{a}_{k-1} \zeta^{j(k-1)})$ 
\end{center}
where $\zeta$ is a primitive $k$th root of unity. Arguing as in the proof of Lemma \ref{detB'} we have 
\begin{center}
    $\det D' = - \frac{n}{12} \displaystyle \prod_{\chi} \frac{1}{4}B_{2,\chi} $ 
\end{center}
where the product is over even, non-principal Dirichlet characters modulo $p$ and $\chi^{k} = 1$, and $B_{2, \chi}$ is a generalised Bernoulli number as defined in the previous section. 

To relate the determinant of $D$ and $D'$, we apply the following row and columnoperations to $D'$:
\begin{itemize}
    \item[1.] subtract the last row from all other rows;
    \item[2.] add all columns to the last one;
\end{itemize}
to obtain the following matrix:
\begin{equation*}
 \begin{psmallmatrix}
    \Tilde{a}_{0} - \Tilde{a}_{k-1} &  \Tilde{a}_{1} - \Tilde{a}_{0} & \ldots  & \Tilde{a}_{k-2} - \Tilde{a}_{k-3} & 0  \\
    \Tilde{a}_{1} - \Tilde{a}_{k-1} & \Tilde{a}_{2} - \Tilde{a}_{0} & \ldots &  \Tilde{a}_{k-1} - \Tilde{a}_{k-3} & 0 \\
    \vdots & \vdots & \ddots & \vdots \\
    \Tilde{a}_{k-2} - \Tilde{a}_{k-1} & \Tilde{a}_{k-1} - \Tilde{a}_{0} & \ldots & \Tilde{a}_{k-4} - \Tilde{a}_{k-3} & 0 \\
    \Tilde{a}_{k-1} & \Tilde{a}_{0} & \ldots & \Tilde{a}_{k-2} & -\frac{n}{12}
    \end{psmallmatrix}.
\end{equation*}
Expanding by the last column gives  
$$ - \frac{n}{12}\det D  = \det D' =  - \frac{n}{12} \displaystyle \prod_{\chi} \frac{1}{4}B_{2,\chi}  $$
and our claim follows. 
 \end{proof}

 \subsection{The Subgroup \texorpdfstring{$\mathcal{C}_{H}^{\Q}(p)$}{a}}
 Consider the subgroup $\mathcal{C}_H^\Q(p)$ of the cuspidal groups generated by degree zero divisors that are supported on cusps and fixed by the Galois action. It follows that 
\begin{center}
     $\mathcal{C}_{H}^{\Q}(p) =  \Divv_{\cusps}^{\Q}(X_{H}(p))/ \langle \divv (f)  \  \vert \ f \in \mathcal{F}_{H}^{\Q}(p) \rangle $ \end{center}
 where $\Divv_{\cusps}^{\Q}(X_{H}(p))$ is the group of degree zero divisors on $X_{H}(p)$, which are supported on cusps and fixed by the Galois action, and  $\mathcal{F}_{H}^{\Q}(p)$ is the subgroup of modular units with divisors in $\Divv_{\cusps}^{\Q}(X_{H}(p))$. Note that $\Divv_{\cusps}^{\Q}(X_{H}(p))$ consists of  divisors of the form 
 \begin{center}
     $\displaystyle \sum_{i=0}^{k-1}a_{i}P_{i} + a(Q_{0} + \ldots +Q_{k-1})$
 \end{center}
 where $a_{i}, a \in \Z$ and $\displaystyle \sum_{i=0}^{k-1}a_{i} + ka =0$. We begin by finding an explicit basis for $\mathcal{F}_{H}^{\Q}(p)$.

 \begin{thm}
 Let $f(\tau) = \displaystyle \prod_{i=0}^{k-1}G_{i}(\tau)^{g_{i}}$ with $g_{i} \in \Z$. Then $f \in \mathcal{F}_{H}^{\Q}(p)$ if and only if the exponents $g_{i}$ satisfy 
 \begin{itemize}
    \item[(i)] $\displaystyle \frac{n}{k}\sum_{i=0}^{k-1} g_{i} \equiv 0 \mod{12};$ 
    \item[(ii)] $\displaystyle \sum_{i=0}^{k-1} g_{i}  \equiv  0 \mod{2}.$ 
\end{itemize}
Moreover, up to scaling, all elements of $\mathcal{F}_{H}^{\Q}(p)$ are of the above form. 
\end{thm}
\begin{proof}
Suppose $f$ is as above with  exponents satisfying the given congruences. Then $f$ is a modular unit on $X_{H}(p)$ by Theorem \ref{XHmainthm}. The divisor of such a function has order 
\begin{center}
    $ \frac{1}{12} \frac{n}{k} \displaystyle \sum_{i=0}^{k-1}g_{i} \in \Z $
\end{center}
at any cusp $Q_{i}$, and thus it is 
fixed by the Galois action and hence  $f \in \mathcal{F}_{H}^{\Q}(p)$. 
On the contrary, if $f$ is a modular unit on $\mathcal{F}_{H}^{\Q}(p)$, then the exponents must satisfy the stated congruences by Theorem \ref{XHmainthm}.
Any element of $ \mathcal{F}_{H}^{\Q}(p)$ is a modular unit on $X_{1}(p)$, with divisor fixed by the action of Galois, and hence by Proposition \ref{F1pQbasis}, it is necessarily of the form $f \left( \tau \right) = \displaystyle \prod_{i=0}^{n-1} E_{i} \left( \tau \right)^{e_{i}}$, with $e_{i} \in \Z $ satisfying 
\begin{itemize}
    \item $ \displaystyle \sum_{i=0}^{n-1} e_{i} \alpha^{2i} \  \equiv 0  \mod{p};$
    \item $ \displaystyle \sum_{i=0}^{n-1} e_{i} \equiv 0 \mod{12}.$
\end{itemize}
Applying Proposition \ref{expsH} to this, we conclude that $f$ is as above.

\end{proof}
As in the previous subsection, finding a basis for $\mathcal{F}_{H}^{\Q}(p)$, modulo constants,  reduces to finding a basis for the set of integer solutions to the necessary and sufficient congruence conditions proved in the previous result. 
\begin{cor}
  The $k$ functions 
  \begin{center}
      $ \frac{G_{0} \left( \tau \right)}{G_{k-1} \left( \tau \right)}, \frac{G_{1} \left( \tau \right)}{G_{k-1} \left( \tau \right)}, \ldots, \frac{G_{k-2} \left( \tau\right)}{G_{k-1} \left( \tau \right)}, G_{k-1} \left( \tau \right)^{R}$
  \end{center}
  form a basis for $\mathcal{F}_{H}^{\Q} \left( p \right)$, modulo constants.
\end{cor}
Recall that $R = \lcm(2, \frac{12}{\gcd(12, \frac{n}{m})}$.

         \begin{proof}[Proof of Theorem \ref{thm3}]
          As detailed previously, computing the order $\mathcal{C}_{H}^{\Q}(p)$ reduces to computing the determinant of the matrix representing the divisors of the basis constructed in the previous result. 
          
          We take 
          $$P_{0} - P_{k-1}, \ldots, P_{k-2} -P_{k-1}, kP_{k-1} - ( Q_{0} + \ldots + Q_{k-1})$$ 
          as a basis for the group of degree $0$, cuspidal divisors fixed we Galois  on $X_{H}(p)$. It follows that   
        
          \begin{center}
              $h_{H}^{\Q}(p) = \vert \det(M) \vert$
          \end{center}
        where $M$ is the $k \times k $ matrix
        \begin{equation*}
M  = 
 \begin{psmallmatrix}
    \Tilde{a}_{0} - \Tilde{a}_{k-1} &  \Tilde{a}_{1} - \Tilde{a}_{0} & \ldots  & \Tilde{a}_{k-2} - \Tilde{a}_{k-3} & 0 \\
    \Tilde{a}_{1} - \Tilde{a}_{k-1} & \Tilde{a}_{2} - \Tilde{a}_{0} & \ldots &  \Tilde{a}_{k-1} - \Tilde{a}_{k-3} &  0 \\
    \vdots & \vdots & \ddots & \vdots  & \vdots \\
    \Tilde{a}_{k-2} - \Tilde{a}_{k-1} & \Tilde{a}_{k-1} - \Tilde{a}_{0} & \ldots & \Tilde{a}_{k-4} - \Tilde{a}_{k-3} & 0 \\
    \star & \star & \ldots & \star & \frac{Rn}{12k}
 \end{psmallmatrix}.
  \end{equation*}        
  Expanding by the last row and using the same arguments as in the proof of Theorem \ref{thm2} we get
  \begin{center}
      $ \vert \det(M) \vert  = \frac{Rn}{12k} \displaystyle \prod_{\chi} \frac{1}{4}B_{2, \chi} $
  \end{center}
  and the result follows from the definition of $R$ and $n$. 
  \end{proof}
  
\subsection{The Rational Cuspidal Group \texorpdfstring{$\mathcal{C}_{H}(p)(\Q)$}{a}} 
The subgroups considered in the previous subsections satisfy the clear containment
\begin{center}
    $\mathcal{C}_{H}^{\infty}(p) \subseteq \mathcal{C}_{H}^{\Q}(p) \subseteq \mathcal{C}_{H}(p)(\Q) \subseteq J_{H}(p)(\Q)_{\tors}$. 
\end{center}
In the case of $X_{1}(p)$, as proved in \cite{lupoian2025modularunitsx1p}, the three rational cuspidal group are in fact equal. In our setting, the first containment is, in general, strict. 

\begin{prop} \label{specialk}
The rational cuspidal groups $\mathcal{C}^{\infty}_{H}(p)$ and $\mathcal{C}_{H}^{\Q}(p)$ are equal if and only if
\begin{center}
    $k \in \{ \frac{p-1}{4}, \frac{p-1}{6}, \frac{p-1}{12}, \frac{p-1}{24} \} \cap \Z$.
    \end{center}
\end{prop}
\begin{proof}
 The claim follows from noting that $\num( \frac{p-1}{12k}) =1$ if and only if $k$ is as stated. 
\end{proof}

The second containment is an equality, as in the case of $X_1(p)$. 
\begin{prop}
$\mathcal{C}^{\Q}_{H} \left( p\right) = \mathcal{C}_{H} \left( p \right) \left( \Q \right)$    
\end{prop}

\begin{proof}
  By definition, there is a short exact sequence 
 \begin{center}
     $ 0 \longrightarrow \text{div} \left( \mathcal{F}_{H} \left( p \right)  \right)  \longrightarrow \text{Div}^{0}_{\text{cusps}} \left( X_{H} \left( p \right)  \right) \longrightarrow  \mathcal{C}_{1} \left( p \right) \longrightarrow 0 $
 \end{center}
 of $\text{Gal} \left( \Bar{ \Q } / \Q \right)$-modules. We take Galois cohomology to obtain 
\begin{center}
   $ \text{Div}^{0}_{\text{cusps}} \left( X_{H} \left(  p \right)  \right) \left( \Q \right) \longrightarrow \mathcal{C}_{ H }\left( p\right) \left( \Q \right) \longrightarrow H^{1} \left( \text{Gal}  \left( \Bar{ \Q} / \Q  \right), \text{div} \left( \mathcal{F}_{H} \left( p\right) \right) \right)$.
    \end{center}
Our claim follows if we show that $ H^{1} \left( \text{Gal}  \left( \Bar{ \Q} / \Q  \right), \text{div} \left( \mathcal{F}_{H} \left( p\right) \right) \right) = 0$. This holds since the absolute Galois group acts through a subgroup of  $\text{Gal} \left( \Q \left( \zeta_{p} \right)^{+} / \Q \right)$, which is a finite cyclic subgroup.
\end{proof}
 This completes the proof of Theorem \ref{thm4}. 

\section{Structure of the cuspidal group} \label{structuresection}

In this section we use the bases of $\mathcal{F}_H(p)$, $\mathcal{F}_H^\Q(p)$, $\mathcal{F}_H^\infty(p)$ and $\mathcal{F}_H^0(p)$ to investigate the structure of $\mathcal{C}_H(p)$ abstractly. 

Recall that $\mathcal{C}_H^\infty(p)$ (resp. $\mathcal{C}_H^0(p)$) denotes the group supported on $P_0, \ldots, P_{k-1}$ (resp. on $Q_0, \ldots, Q_{k-1}$).
\begin{lem}
  Let $p \ge 5$ be prime and $H$ a proper subgroup of $\left( \Z/p
  \Z \right)^*$ such that $-1 \in H$. Then $ \mathcal{C}_H^0(p) \cap \mathcal{C}_H^\infty(p) = \{ 0 \}$. 
\end{lem}
\begin{proof}
   Suppose otherwise. That is, there exist non-principal divisors $D_1 \in \Divv_{0-\text{cusps}}^0(X_H(p))$ and $D_2 \in \Divv_{\infty-\text{cusps}}^0(X_H(p))$ and a modular unit $f$ such that 
   $$  \divv f = D_1 - D_2.$$
By the Manin-Drinfeld theorem, both $D_1$, $D_2$ are torsion, and we note that they necessarily have the same order $n$. More precisely, $n$ is the minimal positive integer such that 
$$ nD_1 = \divv f_1 \ \text{and} \ nD_2 = \divv f_2 $$
for some $f_1 \in \mathcal{F}_H^0(p)$ and $f_2 \in \mathcal{F}_H^\infty(p)$. Therefore, $f^n = f_1 \times f_2^{-1}$. By Theorem \ref{FHpBasis}
and Corollaries \ref{FHinf}, \ref{F1pinfbasis} there exists positive integers $a_i, b_j, c_l, d_m $ such that, up to  scaling, 
\begin{align*}
   & f(\tau) =  \displaystyle \prod_{i= 0}^{k-1}\left(\frac{G_i(\tau)}{H_{k-1}(\tau)^p} \right)^{a_i} \times \displaystyle \prod_{i= 0}^{k-2}\left(\frac{H_i(\tau)}{H_{k-1}(\tau)^p} \right)^{b_i} \times H_{k-1}(\tau)^{Rb_{k-1}};\\
    & f_1(\tau) =  \displaystyle \prod_{i= 0}^{k-2}\left(\frac{G_i(\tau)}{G_{k-1}(\tau)} \right)^{c_i};\\
     & f_2(\tau) =  \displaystyle \prod_{i= 0}^{k-2}\left(\frac{H_i(\tau)}{H_{k-1}(\tau)} \right)^{d_i};
\end{align*}
and $\gcd_{0 \le i \le k-2}(c_i) = 1 =  \gcd_{0 \le i \le k-2}(d_i)$ since we are assuming that the order $n$ is the order, and hence minimal. We re-write the expressions for $f_1$ and $f_2$ in terms of the basis for $\mathcal{F}_H(p)$
\begin{align*}
    & f_1(\tau) =  \displaystyle \prod_{i= 0}^{k-2}\left(\frac{G_i(\tau)}{H_{k-1}(\tau)^p}  \times \frac{H_{k-1}(\tau)^p}{G_{k-1}(\tau)}\right)^{c_i};\\
     & f_2(\tau) =  \displaystyle \prod_{i= 0}^{k-2}\left(\frac{H_i(\tau)}{G_{k-1}(\tau)^p} \times \frac{G_{k-1}(\tau)^p}{H_{k-1}(\tau)} \right)^{d_i}.
\end{align*}

As this is a basis for $\mathcal{F}_H(p)$, up to constants, equating exponents we get
\begin{itemize}
    \item $na_i = c_i$ for $0 \le i \le k-2$; 
    \item $ -\sum_{i=0}^{k-2} c_i = n a_{k-1}$;
    \item $ d_{i} = - nb_i$ for $0 \le i \le k-2 $;
    \item $Rb_{k-1} =0$.
\end{itemize}
In particular, we note that $n \vert c_i, d_i$ for all $i$, which contradicts the co-primality imposed by $n$ being minimal. 
\end{proof}

Let $D_{H,p}$ be the divisors supported on cusps and fixed by the Galois action such that 
$$ \num \left( \frac{p-1}{12k} \right) D_{H,p} = \divv G_{k-1}(\tau)^R.$$
Note that $[D_{H,p}] \in \mathcal{C}_H^\Q(p)$ is an element of order the numerator of $\frac{p-1}{12k}$ by definition. 

\begin{lem}
With notation as above, $\langle [ D_{H,p}] \rangle \cap \mathcal{C}_H^\infty(p) = \{ 0 \}$. 
\end{lem}
\begin{proof}
Suppose otherwise, that is, there exists $D \in \Div_{\infty-\text{cusps}}(X_H(p))$ such that 
$$ D_{H,p} \sim D$$
and notably, $D$ has order $n = \num \left( \frac{p-1}{12k} \right)$. Thus, there exists $ f \in \mathcal{F}_H^\infty(p)$ such that 
$$ n D= \divv f.$$
Equivalently, using the basis of Corollary \ref{FHinf} we find that there exist $a_0, \ldots, a_{k-2}$ such that 
$$ f( \tau) = \displaystyle \prod_{i= 0}^{k-2}\left(\frac{G_i(\tau)}{G_{k-1}(\tau)} \right)^{a_i}.$$
Noting that $f/G_{k-1}^R$ are necessarily constant, using the basis of Theorem \ref{basis} we find that $a_i =0$ for all $i$; a contradiction. 
\end{proof}

By symmetry we also have the following. 
\begin{lem}
  With notation as above, $\langle [ D_{H,p}] \rangle \cap \mathcal{C}_H^0(p) = \{ 0 \}$.   
\end{lem}

The above lemmata, Theorem \ref{cuspidalclassnumber} and the fact that $[D_{H,p}]$ has order $\num \left( \frac{p-1}{12k} \right)$ completes the proof of Theorem \ref{thm4}.

\section{Cuspidal Group Computations} \label{comps}

The explicit bases of modular unit on $X_{H}(p)$ constructed in this paper allow for efficient computations of $\mathcal{C}_H(p)$, $\mathcal{C}_H(p)(\Q)$ and $\mathcal{C}_H^{\infty}(p)$.  We compute all groups for $13 \le p \le 97$ and all possible $H$.  As $p$ is prime and therefore $\left(\Z/p\Z\right)^*$ is cyclic,  $H$ is uniquely determined by the index $k =[ \left(\Z/p\Z \right)^* : H]$. Similar calcualtions for the curve  $X_1(p)$ can be found in \cite{lupoian2025modularunitsx1p}.

The \texttt{Magma} code used for these computations is publicly available at:
\begin{center}
    \href{ https://github.com/ElviraLupoian/CuspidalGroups}{ https://github.com/ElviraLupoian/CuspidalGroups}
\end{center}

\begin{longtable}{|c|c|c| } 
\hline
$p$ & $k$ & $\mathcal{C}_{H} \left( p\right)$ \\
 \hline 
 \SetCell{2}{}{$13$} & $2$ & $1$ \\
 & $3$ & $1$ \\
 \hline 
 \SetCell{2}{}{$17$}
 & $2$ & $\Z / 2 \Z \times \Z / 4 \Z$ \\
  & $4 $ & $\Z / 2 \Z \times \Z / 8 \Z $ \\
 \hline 
\SetCell{1}{}{$19$} 
& $3$ & $ \Z / 9 \Z $ \\
 \hline 
 \SetCell{2}{}{$29$} 
 & $2$ & $ \Z / 3 \Z \times \Z / 21 \Z $ \\
 & $7$ & $ \Z  / 2 \Z \times \Z / 2 \Z \times  \Z / 2 \Z \times \Z / 2 \Z \times \Z / 86 \Z \times \Z / 4214 \Z $ \\
 \hline 
 \SetCell{2}{}{$31$}
 & $3$ &  $ \Z /2 \Z  \times \Z / 2 \Z \times \Z / 14 \Z \times \Z / 70 \Z $ \\
 & $5$ & $\Z / 5 \Z \times \Z / 55 \Z \times \Z /275 \Z$ \\
 \hline 
\SetCell{4}{}{$37$}
& $2$ & $ \Z / 5 \Z \times \Z / 15 \Z $ \\
& $3$ &  $ \Z / 7 \Z \times \Z /63 \Z $\\
& $6$ & $ \Z / 1295 \Z \times \Z / 11655 \Z  $\\
& $9$ & $ \Z / 230223 \Z \times \Z / 2072007 \Z $ \\ 
\hline 
\SetCell{4}{}{$41$}
& $2$ & $ \Z / 8 \Z \times \Z / 40 \Z $ \\
&  $4$ &  $ \Z / 208 \Z \times  \Z / 1040 \Z  $ \\
 & $5$ &  $ \Z / 431 \Z  \times  \Z / 21550 \Z $ \\
 & $ 10$ &  $  \Z / 3313528 \Z  \times  \Z / 82838200 \Z $\\
  \hline 
  \SetCell{2}{}{$43$} 
  &  $3$ & $ \Z / 2\Z  \times  \Z / 2\Z  \times  \Z / 38\Z  \times  \Z / 266 \Z $ \\
&  $7$ &   $ \Z / 13427 \Z  \times  \Z / 657923 \Z $ \\
 \hline
 \SetCell{2}{}{$53$} 
 & $2$ & $ \Z / 7\Z \times  \Z / 91 \Z $ \\
 &  $13$ &  $ \Z / 36562334719\Z  \times  \Z / 6179034567511 \Z $\\ \hline 
  \SetCell{6}{}{$61$} 
&  $2$ & $ \Z / 11 \Z  \times  \Z / 55 \Z  $ \\
   & $3$ & $ \Z / 133 \Z  \times  \Z / 665 \Z $\\
    &  $5$ &  $ \Z / 2801 \Z  \times  \Z / 70025 \Z $ \\
    &  $6$ &  $\Z / 7 \Z  \times  \Z / 7\Z  \times  \Z / 45353 \Z  \times  \Z / 226765  \Z $ \\
      & $10$ & $ \Z / 11 \Z \times  \Z / 11 \Z  \times  \Z / 64117691 \Z  \times \Z / 1602942275 \Z  $  \\
     &  $15$ & $ \Z / 191562660470993 \Z  \times  \Z / 4789066511774825 \Z $ \\

    \SetCell{2}{}{$67$} &  $3$ &  $ \Z /193 \Z \times  \Z /2123 \Z $ \\
   & $11$ &  $ \Z/15145988089 \Z  \times  \Z /1832664558769 \Z $ \\
    \hline 
    \SetCell{2}{}{$71$} & $5 $ &  $ \Z/6541 \Z \times  \Z/1144675 \Z$ \\
     &  $7$ &  $ \Z/1436569 \Z  \times  \Z /351959405 \Z $\\
     \hline 
     \SetCell{7}{}{$73$} & $2$ & $ \Z/22 \Z  \times  \Z/66 \Z $\\
&  $3$ & $ \Z/79 \Z  \times  \Z/1422 \Z $ \\
&  $4$ & $ \Z/2\Z \times  \Z/2\Z \times  \Z/2\Z \times  \Z/2\Z \times  \Z/1958\Z \times  \Z / 5874\Z$ \\
&  $6$ &  $ \Z/418858 \Z \times  \Z/3769722\Z $ \\
&  $9$ &  $\Z / 792000201\Z \times  \Z/14256003618 \Z$ \\
&  $12$ &  $ \Z/2 \Z  \times  \Z/2\Z \times \Z/2\Z \times  \Z/2\Z \times \Z/891586583954 \Z \times  \Z/8024279255586 \Z$ \\
&  $18$ &  $\Z/48697667861546691966 \Z \times  \Z/438279010753920227694 \Z $\\
   \hline 
   \SetCell{2}{}{$79$} & $3$ &  $ \Z/199 \Z  \times  \Z/2587 \Z $\\
   & $13$ & $ \Z/521 \Z  \times  \Z/521 \Z  \times  \Z/92465700913 \Z  \times  \Z/15626703454297 \Z $\\
   \hline
   \SetCell{4}{}{$89$} &  $2$ &  $ \Z/26\Z \times  \Z/286 \Z$ \\
  &  $4$ &  $ \Z/2 \Z \times  \Z/2\Z \times  \Z/2\Z \times  \Z/2 \Z \times \Z/4810 \Z \times  \Z/52910 \Z $\\
 &  $11$ &  $ \Z/1057105915471 \Z \times \Z/255819631543982 \Z$ \\
 & $22$ &  $ \Z/422020141756147446149702662 \Z \times  \Z/51064437152493840984114022102 \Z $ \\
   \hline 
   \SetCell{9}{}{$97$} & $2$ & $ \Z/34 \Z  \times  \Z/136 \Z $\\
  &  $3$  & $ \Z/367 \Z  \times  \Z/2936 \Z $ \\
  & $4$  & $\Z /5066 \Z  \times  \Z/40528 \Z $\\
  & $6$  & $ \Z/5253238 \Z \times \Z/21012952 \Z$ \\
  &  $8$ & $ \Z/748192474 \Z \times  \Z/11971079584 \Z $ \\
  & $12$ & $ \Z/510339217026614 \Z \times  \Z/4082713736212912 \Z $\\
    & $16$ & $ \Z/87358672567862299316 \Z  \times  \Z/1397738761085796789056 \Z$ \\
   & $24$ & $\Z / 5\Z \times  \Z /5\Z \times  \Z/2046224433027098340725813535710 \Z \times  $ \\ & & $ \times \Z/32739590928433573451613016571360 \Z$ \\
   \hline 
 
\caption{$\mathcal{C}_{H} \left( p\right)$ for $13 \le p \le 97$ } 
\label{tab:one}
\end{longtable}

\begin{longtable}{ |c|c|c|c|}
    \hline
$p$ & $k$ & $ \mathcal{C}^{\infty}_{H} \left( p\right) $ &  $\mathcal{C}_{H} \left( p\right) \left( \Q \right)$ \\
\hline
\SetCell{2}{}{$13$} &  $2$ & $1$  & $1$ \\
& $3$ & $1$ & $1$ \\
\hline 
\SetCell{2}{}{$17$} & $2$ &  $\Z / 2\Z $ & $ \Z / 2\Z \times  \Z/ 2\Z$ \\
&  $4$ & $ \Z / 4\Z$ &  $ \Z/ 4 \Z $\\
\hline 
$19$ &  $3$ &   $\Z/ 3\Z $ & $ \Z/3 \Z $\\
\hline 
\SetCell{2}{}{$29$} &  $2$ & $\Z/3\Z $ & $ \Z /21 \Z $\\
& $7$ &  $\Z/ 2\Z \times  \Z / 2\Z \times  \Z/602\Z $ & $ \Z / 2\Z \times  \Z/2 \Z \times  \Z/ 602 \Z $ \\
\hline
\SetCell{2}{}{$31$} & $3$ &  $\Z /2 \Z \times  \Z/ 14 \Z $ &   $ \Z / 2 \Z  \times  \Z/ 70 \Z $ \\
&  $5$ &  $\Z / 5\Z \times  \Z/ 55\Z $ & $ \Z / 5\Z \times  \Z / 55\Z $ \\
\hline 
\SetCell{4}{}{$37$} & $2$ &  $\Z/ 5\Z$ &  $\Z/15\Z$ \\
& $3$ &  $\Z/21 \Z $ & $ \Z/21 \Z$ \\
& $6$ & $\Z/3885 \Z $  &  $ \Z/3885 \Z $ \\
&  $9$ &  $\Z/690669\Z$ & $\Z/690669\Z$\\
\hline 
\SetCell{4}{}{$41$} &  $2$ & $\Z / 8 \Z$ & $ \Z / 40 \Z$ \\
 & $4$ &  $ \Z/208\Z$ &  $\Z/1040 \Z$\\
&  $5$ & $\Z/2155\Z$ & $ \Z/4310 \Z$\\
& $10$ &  $\Z/16567640\Z$ &  $\Z/16567640\Z$ \\
\hline 
\SetCell{2}{}{$43$} &  $3$ &  $ \Z/ 2\Z \times  \Z/38\Z$ &  $\Z/2\Z \times  \Z/266\Z$ \\
&  $7$ &  $\Z/93989 \Z$ & $\Z/93989 \Z$ \\
\hline
\SetCell{2}{}{$53$} &  $2$ &  $\Z/ 7\Z$ &  $ \Z/91 \Z $ \\
 & $13$ & $\Z /475310351347 \Z $ & $ \Z/475310351347 \Z$ \\
 \hline 
 \SetCell{6}{}{$61$} &  $2$ & $\Z/11\Z$ & $\Z/55\Z$ \\
&  $3$ &  $\Z/133\Z $ & $ \Z/ 665\Z$ \\
&  $5$ &  $\Z/14005\Z$ & $ \Z/14005 \Z$ \\
&  $6$ & $ \Z/7\Z \times  \Z/45353 \Z $ & $ \Z/7\Z \times \Z/226765\Z$ \\
&  $10$ &  $ \Z/11\Z \times  \Z/320588455 \Z $ &  $ \Z/11\Z \times  \Z /320588455 \Z$ \\
&  $15$ &  $\Z/957813302354965\Z $ & $ \Z/957813302354965 \Z$ \\
\hline
\SetCell{2}{}{$67$} &  $3$ &  $\Z/193\Z$ &  $\Z/2123 \Z$ \\
& $ 11$ &  $ \Z/166605868979\Z $ & $ \Z/166605868979\Z$ \\
\hline 
\SetCell{2}{}{$71$} & $5$ & $ \Z/32705 \Z$ & $\Z/228935\Z$ \\
& $7$ & $\Z/10055983 \Z$ & $ \Z/50279915\Z $\\
\hline 
\SetCell{8}{}{$73$} & $2$ & $ \Z/22\Z$ & $\Z/66\Z$ \\
&  $3$ & $\Z/237\Z$ &  $\Z/474\Z$ \\
& $4$ &   $\Z/2\Z \times  \Z/2\Z \times  \Z/1958\Z$ &  $\Z/2\Z \times  \Z/2\Z \times  \Z/5874\Z$ \\
& $6$ & $ \Z/1256574\Z $ & $ \Z/1256574\Z$ \\
& $9$ &   $ \Z/2376000603\Z$ & $\Z/4752001206\Z$ \\
&  $12$ &  $\Z/2\Z \times  \Z/ 2\Z $  &  $\Z/2\Z \times  \Z/2\Z $ \\ & & 
$\times \Z/2674759751862\Z$ & $ \times  \Z/2674759751862\Z$ \\
& $18$ & $ \Z/146093003584640075898\Z$ &  $\Z/146093003584640075898\Z$ \\
\hline 
  \SetCell{3}{}{$79$} & $3$ &  $\Z/199\Z$ &  $\Z/2587\Z$\\
& $13$ &  $\Z/521\Z $ &  $\Z/521\Z $  \\
 &  & $\times  \Z/1202054111869\Z$ &  $ \times  \Z/1202054111869\Z$ \\
    \hline 
  \SetCell{4}{}{$89$} & $2$ &  $\Z/26\Z$ &  $\Z/286\Z$ \\
 &  $4$ &  $\Z/2\Z \times  \Z/2\Z \times  \Z/4810\Z$ &  $\Z/2\Z \times  \Z/2\Z \times  \Z/52910 \Z$ \\
 & $11$ &  $\Z/11628165070181\Z$ & $\Z/23256330140362\Z$ \\
&  $22$ &  $\Z/4642221559317621907646729282\Z$ & $\Z/4642221559317621907646729282\Z$ \\
\hline 
\SetCell{8}{}{$97$} &  $2$ &   $\Z/34\Z $ & $ \Z/2\Z \times  \Z/68\Z$ \\
& $3$ & $\Z/367\Z$ &  $\Z/2936\Z$ \\
& $4$ & $\Z/10132\Z$ & $ \Z/2\Z \times  \Z/10132\Z$ \\
&  $6$ &  $\Z/5253238\Z$ &  $\Z/2\Z \times  \Z/10506476\Z$\\
&  $8$ &   $\ Z/2992769896\Z $ & $ \Z/2992769896\Z$ \\
& $12$ & $\Z/1020678434053228\Z$ &  $\Z/2\Z \times  \Z/1020678434053228 \Z$ \\
& $16$ &  $\Z/349434690271449197264\Z$ &  $\Z/349434690271449197264\Z$\\
& $24$ &  $\Z/5\Z \times  \Z/8184897732108393362903254142840\Z$ & $ \Z/5\Z \times  \Z/8184897732108393362903254142840\Z$ \\
  \hline 
\caption{$\mathcal{C}_{H}^{\infty} \left( p\right)$ and $\mathcal{C}_{H} \left( p\right) \left( \Q \right)$ } 
\end{longtable}

\bibliographystyle{amsplain} 
\bibliography{ref}

@book{KL,
  title={Modular Units},
  author={Kubert, Daniel S and Lang, Serge},
  year={1981},
  publisher={Springer}
}

@misc{yang2007modular,
      title= {{Modular Units and Cuspidal Divisor Class Groups of $X_1(N)$}}, 
      author={Yifan Yang},
      year={2007},
      eprint={0712.0629},
      archivePrefix={arXiv},
      primaryClass={math.NT},
}

@book{yu1980cuspidal,
  title={{A cuspidal class number formula for the modular curves $X_{1} (N)$}},
  author={Yu, Jing},
  year={1980},
  publisher={Yale University}
}

@article{takagix1p,
  title={{Cuspidal class number formula for the modular curves $X_{1} (p)$}},
  author={Takagi, Toshikazu},
  journal={Journal of Algebra},
  volume={151},
  number={2},
  pages={348--374},
  year={1992},
  publisher={Academic Press}
}

@article {Manin,
    AUTHOR = {Manin, Ju. I.},
     TITLE = {{Parabolic points and zeta functions of modular curves}},
   JOURNAL = {Izv. Akad. Nauk SSSR Ser. Mat.},
  FJOURNAL = {Izvestiya Akademii Nauk SSSR. Seriya Matematicheskaya},
    VOLUME = {36},
      YEAR = {1972},
     PAGES = {19--66},
      ISSN = {0373-2436},
   MRCLASS = {14G10 (10D15 14H25)},
  MRNUMBER = {0314846},
MRREVIEWER = {K.-B. Gundlach},
}

@article {Drinfeld,
    AUTHOR = {Drinfeld, V. G.},
     TITLE = {{Two theorems on modular curves}},
   JOURNAL = {Funkcional. Anal. i Prilo\v{z}en.},
  FJOURNAL = {Akademija Nauk SSSR. Funkcional\cprime nyi Analiz i ego Prilo\v{z}enija},
    VOLUME = {7},
      YEAR = {1973},
    NUMBER = {2},
     PAGES = {83--84},
      ISSN = {0374-1990},
   MRCLASS = {14G10 (10D15 14H25)},
  MRNUMBER = {0318157},
MRREVIEWER = {K.-B. Gundlach},
}

@article{newman1957construction,
  title={{Construction and application of a class of modular functions}},
  author={Newman, Morris},
  journal={Proceedings of the London Mathematical Society},
  volume={3},
  number={1},
  pages={334--350},
  year={1957},
  publisher={Oxford University Press}
}

@article{takagi2010cuspidal,
  title={{The cuspidal class number formula for certain quotient curves of the modular curve $X_{0}(M)$ by Atkin-Lehner involutions}},
  author={Takagi, Toshikazu},
  journal={Journal of the Mathematical Society of Japan},
  volume={62},
  number={1},
  pages={13--47},
  year={2010},
  publisher={The Mathematical Society of Japan}
}

@inproceedings {Ogg,
    AUTHOR = {Ogg, A. P.},
     TITLE = {{Rational points on certain elliptic modular curves}},
 BOOKTITLE = {Analytic number theory ({P}roc. {S}ympos. {P}ure {M}ath.,
              {V}ol. {XXIV}, {S}t. {L}ouis {U}niv., {S}t. {L}ouis, {M}o.,
              1972)},
     PAGES = {221--231},
 PUBLISHER = {Amer. Math. Soc., Providence, R.I.},
      YEAR = {1973},
   MRCLASS = {14G05 (10D05 14K15)},
  MRNUMBER = {0337974},
MRREVIEWER = {J. W. S. Cassels},
}

@article{yoo2023rational,
  title={{The rational cuspidal divisor class group of $X_{0}(N)$}},
  author={Yoo, Hwajong},
  journal={Journal of Number Theory},
  volume={242},
  pages={278--401},
  year={2023},
  publisher={Elsevier}
}

@article{takagi1997cuspidal,
  title={{The Cuspidal Class Number Formula for the Modular Curves $X_{0}(M)$ withMSquare-Free}},
  author={Takagi, Toshikazu},
  journal={Journal of Algebra},
  volume={193},
  number={1},
  pages={180--213},
  year={1997},
  publisher={Elsevier}
}

@article{Kubert19753,
author = {Kubert, Dan and  Lang, Serge},
journal = {Mathematische Annalen},
pages = {273-285},
title = {{Units in the Modular Function Field. III. Distribution Relations.}},
url = {http://eudml.org/doc/162800},
volume = {218},
year = {1975},
}

@article{takagi1993cuspidal,
  title={{The cuspidal class number formula for the modular curves $X_{1}(p^m)$}},
  author={Takagi, Toshikazu},
  journal={Journal of Algebra},
  volume={158},
  number={2},
  pages={515--549},
  year={1993},
  publisher={Academic Press}
}

@article{takagi1995cuspidal,
  title={{The cuspidal class number formula for the modular curves $X_{1}(3^m)$}},
  author={Takagi, Toshikazu},
  journal={Journal of the Mathematical Society of Japan},
  volume={47},
  number={4},
  pages={671--686},
  year={1995},
  publisher={The Mathematical Society of Japan}
}

@article{takagi2012cuspidal,
  title={{The cuspidal class number formula for the modular curves $X_{1}(2p)$}},
  author={Takagi, Toshikazu},
  journal={Journal of the Mathematical Society of Japan},
  volume={64},
  number={1},
  pages={23--85},
  year={2012},
  publisher={The Mathematical Society of Japan}
}

@article{yang2010structure,
  title={{Structure of the cuspidal rational torsion subgroup of $J_{1}(p^n)$}},
  author={Yang, Yifan and Yu, Jeng-Daw},
  journal={Journal of the London Mathematical Society},
  volume={82},
  number={1},
  pages={203--228},
  year={2010},
  publisher={Wiley Online Library}
}

@article{diamond2005first,
  title={{A First Course in Modular Forms}},
  author={Diamond, F},
  journal={Graduate Texts in Mathematics/Springer-Verlag},
  volume={436},
  year={2005}
}

@article{ling1997q,
  title={{On the $\mathbb{Q}$-rational cuspidal subgroup and the component group of $J_{0}(p^r)$}},
  author={Ling, San},
  journal={Israel Journal of Mathematics},
  volume={99},
  pages={29--54},
  year={1997},
  publisher={Springer}
}

@article{chua1997rational,
  title={{On the rational cuspidal subgroup and the rational torsion points of $J_{0}(pq)$}},
  author={Chua, Seng-Kiat and Ling, San},
  journal={Proceedings of the American Mathematical Society},
  volume={125},
  number={8},
  pages={2255--2263},
  year={1997}
}

@article{lorenzini1995torsion,
  title={{Torsion points on the modular jacobian $ J_{0}(N)$}},
  author={Lorenzini, Dino J},
  journal={Compositio Mathematica},
  volume={96},
  number={2},
  pages={149--172},
  year={1995}
}

@article{rouse2015spaces,
  title={{On spaces of modular forms spanned by eta-quotients}},
  author={Rouse, Jeremy and Webb, John J},
  journal={Advances in Mathematics},
  volume={272},
  pages={200--224},
  year={2015},
  publisher={Elsevier}
}

@misc{lupoian2025modularunitsx1p,
      title={{Modular Units on $X_{1}( p)$ and Quotients of the Cuspidal Group}}, 
      author={Lupoian, Elvira},
      year={2025},
      eprint={2502.04084},
      archivePrefix={arXiv},
      primaryClass={math.NT},
      url={https://arxiv.org/abs/2502.04084}, 
}

@article{ingleton1956rank,
  title={{The rank of circulant matrices}},
  author={Ingleton, Aubrey W},
  journal={Journal of the London Mathematical Society},
  volume={1},
  number={4},
  pages={445--460},
  year={1956},
  publisher={Oxford University Press}
}

@article{PMB_2025____97_0,
     author = {Lupoian, Elvira},
     title = {{Computing the {Cuspidal} {Subgroup} of the {Modular} {Jacobian} $J_{H}(p)$}},
     journal = {Publications math\'ematiques de Besan\c{c}on. Alg\`ebre et th\'eorie des nombres},
     pages = {97--113},
     year = {2025},
     publisher = {Presses universitaires de Franche-Comt\'e},
     doi = {10.5802/pmb.63},
     language = {en},
     url = {https://pmb.centre-mersenne.org/articles/10.5802/pmb.63/}
}

@article{Chen,
author = {Yao-Han Chen},
title = {{Cuspidal $\mathbb{Q}$-Rational Torsion Subgroup of $J(\Gamma)$ of Level $P$}},
volume = {15},
journal = {Taiwanese Journal of Mathematics},
number = {3},
publisher = {Mathematical Society of the Republic of China},
pages = {1305 -- 1323},
year = {2011},
doi = {10.11650/twjm/1500406301},
URL = {https://doi.org/10.11650/twjm/1500406301}
}

@article{guo2023rational,
  title={{The rational cuspidal subgroup of $J_0(p^2M)$  with $M$ squarefree}},
  author={Guo, Jia-Wei and Yang, Yifan and Yoo, Hwajong and Yu, Myungjun},
  journal={Mathematische Nachrichten},
  volume={296},
  number={10},
  pages={4634--4655},
  year={2023},
  publisher={Wiley Online Library}
}

@article{brunault2026bernoulli,
  title={{Bernoulli determinants and cuspidal subgroups}},
  author={Brunault, Fran{\c{c}}ois},
  journal={arXiv preprint arXiv:2607.13536},
  year={2026}
}

@article{takagi2008cuspidal,
  title={{The cuspidal class number formula for the modular curves $X_1(2^{2n+1})$}},
  author={Takagi, Toshikazu},
  journal={Journal of Algebra},
  volume={319},
  number={9},
  pages={3535--3566},
  year={2008},
  publisher={Elsevier}
}

@article{takagi2014ℚ,
  title={{The $\mathbb{Q}$-rational cuspidal group of $J_1(2p)$}},
  author={Takagi, Toshikazu},
  journal={Journal of the Mathematical Society of Japan},
  volume={66},
  number={4},
  pages={1249--1301},
  year={2014},
  publisher={The Mathematical Society of Japan}
}

@book {klimek,
    AUTHOR = {Klimek, Peter Edward},
     TITLE = {{Modular Functions for $\Gamma_1(N)$}},
      NOTE = {Thesis (Ph.D.)--University of California, Berkeley},
 PUBLISHER = {ProQuest LLC, Ann Arbor, MI},
      YEAR = {1975},
     PAGES = {(no paging)},
   MRCLASS = {99-05},
  MRNUMBER = {2940570},
       URL =
              {http://gateway.proquest.com/openurl?url_ver=Z39.88-2004&rft_val_fmt=info:ofi/fmt:kev:mtx:dissertation&res_dat=xri:pqdiss&rft_dat=xri:pqdiss:0291808},
}

@article{sun2010cuspidal,
  title={{Cuspidal class number of the tower of modular curves $X_1(Np^n)$}},
  author={Sun, Hae-Sang},
  journal={Mathematische Annalen},
  volume={348},
  number={4},
  pages={909--927},
  year={2010},
  publisher={Springer}
}

@article{yoo2025rational,
  title={{The rational cuspidal subgroup of $J_0(N)$}},
  author={Yoo, Hwajong and Yu, Myungjun},
  journal={arXiv preprint arXiv:2504.12564},
  year={2025}
}

\end{document}